\documentclass{article}
\usepackage[
  left=2.5cm,
  right=2.5cm,
  top=3cm,
  bottom=3cm
]{geometry}
\usepackage[utf8]{inputenc}
\usepackage{amsmath}
\usepackage{amssymb}
\usepackage{amsthm}
\usepackage{mathtools}
\usepackage{hyperref}
\usepackage{mathrsfs}
\usepackage{bbm}
\usepackage{bm}
\usepackage{comment}
\usepackage{enumerate}
\usepackage{float}
\usepackage{subcaption}
\usepackage{enumitem}
\usepackage{authblk}

\usepackage[capitalize,nameinlink]{cleveref}
\usepackage{graphicx}

\usepackage[
  backend=biber,
  bibstyle=ieee,
  citestyle=numeric,
  isbn=true,
  doi=true,
  sorting=none,
  url=true,
  bibencoding=utf8,
  maxnames=2,
  debug=true
]{biblatex}
\usepackage{indentfirst}

\DeclareMathOperator{\sign}{sign}
\DeclarePairedDelimiter\abs{\lvert}{\rvert}%

\DeclarePairedDelimiter\norm{\lVert}{\rVert}%
\DeclarePairedDelimiter\ceil{\lceil}{\rceil}
\DeclarePairedDelimiter\floor{\lfloor}{\rfloor}

\makeatletter
\let\oldabs\abs
\def\abs{\@ifstar{\oldabs}{\oldabs*}}
\let\oldnorm\norm
\def\norm{\@ifstar{\oldnorm}{\oldnorm*}}

\let\oldfloor\floor
\def\floor{\@ifstar{\oldfloor}{\oldfloor*}}

\let\oldceil\ceil
\def\ceil{\@ifstar{\oldceil}{\oldceil*}}
\makeatother

\newtheorem{theorem}{Theorem}[section]

\newtheorem{lemma}[theorem]{Lemma}

\newtheorem{remark}[theorem]{Remark}

\newcommand{\authcite}[1]{\citeauthor{#1}~\cite{#1}}

\newcommand*{\R}{\mathbb{R}}
\newcommand*{\N}{\mathbb{N}}
\newcommand*{\Z}{\mathbb{Z}}

\newcommand*{\E}{\boldsymbol{\mathcal{E}}}
\newcommand*{\uu}{\bold{u}}
\newcommand*{\vv}{\bold{v}}

\newcommand*{\kk}{\bold{k}}
\newcommand*{\pp}{\bold{p}}
\newcommand*{\qq}{\bold{q}}
\newcommand*{\mm}{\bold{m}}

\newcommand*{\ww}{\bold{w}}
\newcommand*{\jj}{\bold{j}}
\newcommand*{\ii}{\bold{i}}
\newcommand*{\hh}{\bold{h}}
\newcommand*{\xx}{\bold{x}}
\newcommand*{\yy}{\bold{y}}
\renewcommand*{\AA}{\bold{A}}

\newcommand*{\LL}{\boldsymbol{\mathcal{L}}}
\newcommand*{\VV}{\bold{V}}
\newcommand*{\WW}{\bold{W}}
\newcommand*{\CC}{\bold{C}}
\newcommand*{\MM}{\bold{M}}
\newcommand*{\NN}{\bold{N}}
\newcommand*{\PP}{\bold{P}}
\newcommand*{\RR}{\bold{R}}
\newcommand*{\KK}{\bold{K}}
\newcommand*{\EE}{\bold{E}}
\renewcommand*{\SS}{\bold{S}}
\newcommand*{\ff}{\bold{f}}
\newcommand*{\bb}{\bold{b}}

\renewcommand*{\ff}{\bold{f}}
\newcommand*{\xixi}{\boldsymbol{\xi}}
\newcommand*{\GammaGamma}{\boldsymbol{\Gamma}}
\renewcommand*{\Re}{\mathrm{Re}}

\newcommand*{\LandO}{\mathcal{O}}

\newcommand\restr[2]{{
  \left.\kern-\nulldelimiterspace 
  #1 
  \vphantom{\big|} 
  \right|_{#2} 
  }}

\title{Parametric Spectral Submanifolds across Hopf Bifurcations with  Applications to Fluid Dynamics \\
\vspace{10pt}
\large (Submitted to Chaos) }

\author[1]{James King}
\author[1,2]{Bálint Kaszás}
\author[1]{Gergely Buza}
\author[3]{William Jussiau}
\author[1]{George Haller}

\affil[1]{Institute for Mechanical Systems, ETH Zurich, Zurich, Switzerland}
\affil[2]{Department of Mechanical Engineering, Stanford University, Stanford, USA}
\affil[3]{Department of Engineering Cybernetics, NTNU, Trondheim, Norway}

\date{}

\begin{document}

\maketitle

\begin{abstract}
    We investigate the persistence and regularity of spectral submanifolds (SSMs) in high-dimensional parametric dynamical systems undergoing a Hopf bifurcation.
    By analyzing how resonances in the linearized spectrum near bifurcation points limit the existence and smoothness of SSMs, a phenomenon that has been mostly overlooked, we show that low-order Taylor coefficients of the SSM expansion and the associated reduced dynamics persist smoothly through the bifurcation. 
    This analysis generalizes to any local bifurcation and provides a clear estimate of the parameter ranges over which a parametric SSM model can be justified, thus illustrating how globally the model can be extended despite the presence of resonances near criticality.
    We demonstrate these findings on multiple examples, including a data-driven SSM approach to the lid-driven cavity flow. For that problem, we construct a parametric SSM-reduced model that accurately captures the full transition to periodic dynamics and the critical Reynolds number. 
    These results provide a mathematical foundation for robust data- and equation-driven model reduction of fluid flows across bifurcations, enabling an accurate prediction of nonlinear dynamics across critical parameter regimes.
\end{abstract}
\section{Introduction}
The study of bifurcations in nonlinear dynamical systems is fundamental to understanding the emergence of complex nonlinear phenomena, such as (quasi-) periodic motion, transitions between isolated fixed points, and chaotic dynamics. Among these, the supercritical Hopf bifurcation is of particular interest, as it describes the transition from steady-state to periodic dynamics through the birth of a limit cycle - a phenomenon ubiquitous in many engineering applications.

In high-dimensional dynamical systems, such as the Navier-Stokes equations, direct analysis of the full dynamics is often intractable. This motivates model reduction approaches to low-dimensional systems that capture the essential dynamics of the full system, and still retain high predictive accuracy.
A mathematically justified approach to model reduction is to reduce the system to a low-dimensional invariant manifold. 
In the context of bifurcations of fixed points, center manifold reduction has traditionally served as the standard approach for analyzing local bifurcations. By the center manifold theorem (see, e.g., \authcite{guckenheimerNonlinearOscillationsDynamical1984}), there exists a (generally non-unique) invariant extended center manifold tangent to the center subspace of the bifurcating fixed point. While accurate and successful, center manifold reductions are local by nature and only hold for parameter values on an open neighborhood around the bifurcation value. Moreover, precise bounds for such a neighbourhood are often difficult to find, especially if knowledge about the full system is limited.

A more general notion is that of a Spectral Submanifold (SSM), which includes any invariant manifold that is anchored to a fixed point, and is tangent to a spectral subspace of the dynamical system's linearization at that fixed point. Under nonresonance conditions on the spectrum of the linearization, there exists a unique smoothest SSM (called the primary SSM), that acts as the nonlinear equivalent of the spectral subspace itself. Built on related early results of \authcite{cabreParameterizationMethodInvariant2003} for autonomous mappings on Banach spaces and \authcite{haroParameterizationMethodComputation2006} for quasi-periodic mappings, SSM theory has been extended to spectral subspaces of mixed stability type, nonautonomous systems with small or slow forcing, stochastic systems, and nonsmooth dynamics (see \authcite{hallerModelingNonlinearDynamics2025} for a recent review). Based on these results, SSM-reduced models have been constructed and validated in the equation-driven setting (see, e.g., \authcite{ponsioenAutomatedComputationAutonomous2018} and \authcite{jainHowComputeInvariant2022})
and in the data-driven setting (see, e.g., \authcite{cenedeseDatadrivenModelingPrediction2022} and \authcite{axasFastDatadrivenModel2023}).
Specifically in the field of fluid dynamics, they have been used to model various systems including vortex shedding (\authcite{cenedeseDatadrivenModelingPrediction2022}), fluid sloshing experiments (\authcite{cenedeseDatadrivenModelingPrediction2022}), fluid-structure interaction experiments (\authcite{xuDatadrivenModellingRegular2024}), and shear flows (\authcite{kaszasCapturingEdgeChaos2024},\cite{kaszasJointReducedModel2025}). The existence of SSMs has also been shown in the Navier-Stokes equations themselves by \authcite{buzaSpectralSubmanifoldsNavier2024}.
Furthermore, since SSMs depend smoothly on parameters (see \authcite{cabreParameterizationMethodInvariant2003a}), parametric SSM-based models by \authcite{liParametricModelReduction2024}, \authcite{morsyModelReductionSystems2025}, and \authcite{abbascianoDataDrivenModelingGlobal2026} have revealed how SSMs and their reduced dynamics evolve as system parameters vary. For a general overview, we refer to \authcite{hallerModelingNonlinearDynamics2025}.

Local bifurcations in the anchor points of SSMs, however, pose significant challenges for parametric SSM-based reduction, as the nonresonance conditions required for the existence and smoothness of SSMs are guaranteed to fail near bifurcation, assuming the presence of at least one real negative eigenvalue. 
Therefore, a thorough understanding of the persistence and regularity of SSM expansions across such bifurcations is essential for constructing reliable reduced-order models that capture the bifurcation.

In this work, we analyze the behaviour of SSMs in the vicinity of a Hopf bifurcation of its anchor point. We show a generic accumulation of resonances near the bifurcation. We nevertheless obtain the persistence of low-order coefficients in the Taylor expansion of the manifold and the associated reduced dynamics. 
Importantly, these considerations are not limited to Hopf bifurcations - they generalize to any type of bifurcation.
After illustrating our theoretical findings on the canonical Hopf normal form, we perform a similar but data-driven analysis of the lid-driven cavity flow, a classical benchmark in fluid dynamics that exhibits a supercritical Hopf bifurcation as the Reynolds number is increased.

\subsection{Prior Work}

Center manifold reduction has successfully been used to capture Hopf bifurcations in the Navier-Stokes equations. By introducing a parameter depending on the Reynolds number, \authcite{cariniCentremanifoldReductionBifurcating2015} perform a center manifold reduction in the flow past a circular cylinder, which exhibits a supercritical Hopf bifurcation. They use the same method to capture a codimension-two bifurcation of the flow past two side-by-side cylinders and report accurate low-dimensional, reduced dyanamics conjugate to those on the center manifold.

Parameter-dependent center and unstable manifolds are calculated by \authcite{liParametricModelReduction2024} to model the post-Hopf bifurcation dynamics of a cantilevered pipe conveying fluid. To deal with parameter dependence, \authcite{liParametricModelReduction2024} also solve the extended invariance equations satisfied by the parametric invariant manifolds, expanded around different parameter values. In a similar fashion, \authcite{colomboReducedOrderModelling2025} solve the extended invariance equations for the flow past a circular cylinder, expanded around different Reynolds numbers. They also make predictions for a small range of pre-bifurcation parameter values.

In general, such an expansion in parameters is an inherently more local approach than the coefficient interpolation we will perform in this work. Moreover, as we will show, the parameter ranges in which the expansions in \authcite{cariniCentremanifoldReductionBifurcating2015} and \authcite{colomboReducedOrderModelling2025} are valid are limited precisely by resonances within the spectrum of the linearization. 
This explains why \authcite{colomboReducedOrderModelling2025} were not able to calculate expansions centered around pre-bifurcation parameter values that were close to criticality, and why their expansions centered around post-bifurcation parameters become inaccurate pre-bifurcation.

Interpolating SSM coefficients in order to create a data-driven parametric SSM model across bifurcations has already been carried out by \authcite{abbascianoDataDrivenModelingGlobal2026}. In that case however, the anchor point of the SSM remains hyperbolic while local and global bifurcations take place within the SSM away from its anchor point. As a consequence, the nonresonance conditions remain satisfied across the parameter range studied.

As an alternative to interpolation over multiple models obtained at different parameter values, \authcite{morsyModelReductionSystems2025} construct an equation-driven parametric SSM model. Their procedure only requires a single, parameter-dependent calculation, that does not have the locality limitations of parameter-based expansions. This is achieved by projecting the parameter dependence onto a suitable orthogonal polynomial basis, which \authcite{morsyModelReductionSystems2025} then use for uncertainty quantification.

\authcite{delallaveGlobalPersistenceLyapunov2019} prove the parametric persistence of Ljapunov Subcenter Manifolds (LSMs) as SSMs under dissipative perturbations. By taking advantage of the conserved quantity of the unperturbed Hamiltonian system, they are able to show the existence and uniqueness of an SSM on a domain independent of the perturbation parameter, which, as they point out, is not generally possible in the setting of this work. Because \authcite{delallaveGlobalPersistenceLyapunov2019} assume there to be no real eigenvalues in the spectrum of the linearization, the resonances encountered in this work are not present there.

\section{Spectral Submanifolds (SSMs)}
\label{sec:SSM_Theory}

Consider the autonomous parametric nonlinear dynamical system
\begin{align}
    \dot{\xx} = \ff(\xx;\mu), \qquad \xx \in U \subset \R^N, \qquad \mu \in V \subset \R,
    \label{eq:parametric_dynamical_system}
\end{align}
where $U$ and $V$ are open, connected sets and $\ff\in C^{r+1}(U\times V,\R^N)$ for some $r\geq 1$. Assume that for all $\mu \in V$, system \eqref{eq:parametric_dynamical_system} has a fixed point $\xx_0(\mu)$ that depends $C^r$ smoothly on $\mu$. This lets us write \eqref{eq:parametric_dynamical_system} as
\begin{align}
    \dot{\xx} = \AA(\mu)\xx + \ff_0(\xx;\mu), \qquad \AA(\mu)=D_{\xx}\ff(\xx_0(\mu),\mu), \qquad \ff_0(\xx;\mu)=\LandO(\norm{\xx}^2),
    \label{eq:system_around_fixed_point}
\end{align}
where $\AA \in C^{r}(V,\R^{N\times N})$ and $\ff_0 \in C^{r}(U\times V,\R^N)$.
The existence and smoothness of a unique primary SSM, $W(E;\mu)$, tangent to a $d$-dimensional spectral subspace $E(\mu)$ (i.e. the direct sum of some real generalized eigenspaces of $\AA(\mu)$) is shown by \authcite{cabreParameterizationMethodInvariant2003} under the assumptions that
\begin{enumerate}[label=(A\arabic*)]
    \item $\AA(\mu)$ has a uniformly stable spectrum \label{assump:uniformly_stable_spectrum}
    \begin{align}
        \Re\lambda_N(\mu) \leq \dots \leq \Re\lambda_1(\mu) < 0, \qquad \forall \mu \in V.
        \label{eq:uniformly_stable_spectrum}
    \end{align}
    \item The spectral quotient $Q(\mu)$ does not exceed the smoothness of the nonlinearities \label{assump:spectral_quotient}
    \begin{align}
        r\geq Q(\mu)= \left\lfloor \frac{\Re\lambda_{N}(\mu)}{\Re\lambda_{1}(\mu)} \right\rfloor + 1, \qquad \forall \mu \in V.
        \label{eq:spectral_quotient}
    \end{align}
    \item The spectrum of $\AA(\mu)$ is uniformly nonresonant \label{assump:nonresonance_conditions}
    \begin{align}
        \lambda_l(\mu) \neq \sum_{j=1}^d m_j \lambda_{j}(\mu), \ m_j \in \N, \ 2\leq \sum_{j=1}^d m_j \leq Q(\mu), \ \lambda_j(\mu) \in \text{spect}[\AA(\mu)]|_{E(\mu)}, \ \lambda_l(\mu) \notin \text{spect}[\AA(\mu)]|_{E(\mu)}.
        \label{eq:nonresonance_conditions}
    \end{align}
\end{enumerate}
Under assumptions \ref{assump:uniformly_stable_spectrum}-\ref{assump:nonresonance_conditions}, the SSM, $W(E;\mu)$, and the dynamics within $W(E;\mu)$ are both $C^{r-Q-1}$ smooth in $(\xx, \mu)$, where $Q=\sup_{\mu \in V} Q(\mu)$. In the case where $\ff$ is $C^\infty$ or analytic, the SSM and SSM-reduced dynamics are also $C^\infty$ or analytic, respectively.

If the nonresonance conditions \eqref{eq:nonresonance_conditions} are violated for some multi index $\mm \in \N^d$, then we call $\abs{\mm} = \sum_{j=1}^d m_j$ the order of that resonance.

\clearpage
\section{Continuing SSMs through a Hopf bifurcation}
Consider the dynamical system \eqref{eq:parametric_dynamical_system} with $N\geq 3$
and assume that the system undergoes a Hopf bifurcation at $\mu = \mu_0$.
Specifically, assume that
\begin{enumerate}[resume*]
    \item $\AA(\mu)$ has a complex conjugate eigenvalue pair crossing the imaginary axis at $\mu_0$
\begin{align*}
\lambda_{1,2}(\mu) = \alpha(\mu)\pm i \omega(\mu), \qquad \omega(\mu) \neq 0, \alpha(\mu_0) = 0, \qquad \alpha(\mu) < 0, \ \forall \mu < \mu_0, \qquad \alpha(\mu) > 0, \ \forall \mu > \mu_0,
\end{align*}
and the rest of the spectrum has negative real part \label{assump:bifurcating_eigenvalue_pair}
\begin{align}
    \Re \lambda_N(\mu) \leq \dots \leq \Re \lambda_{3}(\mu) < 0, \qquad \forall \mu \in V.
    \label{eq:ordering_eigenvalues}
\end{align}
\end{enumerate}
The linearization of \eqref{eq:parametric_dynamical_system} around $\xx_0(\mu)$,
\begin{align}
    \dot{\xixi} = \AA(\mu)\xixi,
    \label{eq:linearised_system}
\end{align}
has a class $C^{r}$ family of 2D real eigenspaces, $E(\mu)$, corresponding to $\lambda_{1,2}(\mu)$ that we can align using a $\mu$-dependent coordinate change. 
We denote this family by
\begin{align}
    E^{\text{slow}} = \left(\bigcup_{\mu < \mu_0} E^s(\mu)\right) \cup E^c(\mu_0) \cup \left(\bigcup_{\mu > \mu_0} E^u(\mu)\right),
    \label{eq:family_eigenspaces}
\end{align}
where $E^s, E^c, E^u$ denote the slow spectral subspace, the center subspace and the unstable subspace, respectively.
We now turn to constructing invariant manifolds of \eqref{eq:parametric_dynamical_system} that are tangent to these eigenspaces.
\subsection{Invariant manifolds near a Hopf Bifurcation}
By the center manifold theorem, any fixed point $\xx_0(\mu)$ with $\mu>\mu_0$ has a unique 2D, class $C^{r}$, unstable manifold $W^u(\mu)$ that is tangent to $E^u(\mu)$ at $\xx_0(\mu)$. This family of unstable manifolds is also jointly $C^{r}$ in $(\xx, \mu)$.

By the Hopf-Bogdanov theorem (see, e.g. \authcite{guckenheimerNonlinearOscillationsDynamical1984}), the extended system
\begin{align}
    \dot{\xx} &= \ff(\xx;\mu), \nonumber \\
    \dot{\mu} &= 0,
    \label{eq:extended_system}
\end{align}
has a 3D, class $C^{r}$, unique extended center manifold $\hat{W}^c(\mu_0)$ at $(\xx_0(\mu_0),\mu_0)$, tangent to $E^c(\mu_0)\times \R$ and defined on an open neighborhood $U_c$ of $(\xx_0(\mu_0),\mu_0)$. 
The manifold $\hat{W}^c(\mu_0)$ contains all bounded solutions of \eqref{eq:extended_system} on $U_c$, which includes $\xx_0(\mu)$ for $\mu>\mu_0$ (\authcite{vanderbauwhedeCentreManifoldsNormal1989}). Therefore,
\begin{align*}
    \forall \mu_1 > \mu_0: \ \hat{W}^c(\mu_0) \cap \{(\xx,\mu) \colon \mu = \mu_1\} = W^u(\mu_1),
\end{align*}
meaning that we can smoothly continue the center manifold into the unstable manifold locally by varying $\mu$.

Based on Section \ref{sec:SSM_Theory}, we can conclude a similar statement for $\mu<\mu_0$ on the existence of a 2D SSM, whenever the nonresonance conditions \eqref{eq:nonresonance_conditions} hold. However, as we will show in the next section, as soon as $\text{spect}[\AA(\mu)]$ includes a purely real negative eigenvalue outside $E(\mu)$, the nonresonance conditions \eqref{eq:nonresonance_conditions} will fail arbitrarily close to the bifurcation.

\subsubsection{Accumulation of resonances near a Hopf bifurcation}
\label{sec:accumulation_of_resonances_near_Hopf}
Consider again the dynamical system \eqref{eq:parametric_dynamical_system} with $N\geq 3$ undergoing a Hopf bifurcation at $\mu = \mu_0$. Assume further that
\setcounter{enumi}{4}
\begin{enumerate}[resume*]
    \item $\AA(\mu)$ has a purely real negative eigenvalue
    \begin{align*}
        \lambda_{3}(\mu) = \nu(\mu) < 0, \qquad \lambda_{3}(\mu) < \Re \lambda_{1,2}(\mu), \qquad \forall \mu \leq \mu_0.
    \end{align*}
    \item The eigenvalues $\lambda_{1,2,3}(\mu)$ are simple and disjoint from the rest of the spectrum for all $\mu$. \label{assump:simple_disjoint_eigvals}
\end{enumerate}
We want to asses the existence of SSMs over spectral subspaces $E(\mu)$ with
\begin{align*}
    \lambda_{1,2}(\mu) \in \text{spect}[\AA(\mu)]|_{E(\mu)}, \ \lambda_3(\mu) \notin \text{spect}[\AA(\mu)]|_{E(\mu)}. 
\end{align*}
Since $\overline{\lambda_1}(\mu)=\lambda_2(\mu)$, we have
\begin{align*}
    m_1\lambda_1(\mu)+m_2\lambda_2(\mu) \in \R \Leftrightarrow m_1=m_2.
\end{align*}
As a consequence, the nonresonance conditions
\begin{align*}
\lambda_3(\mu) \neq m_1\lambda_1(\mu)+m_2\lambda_2(\mu),
\qquad 2\leq m_1+m_2 \leq \left\lfloor \frac{\Re\lambda_{3}(\mu)}{\Re\lambda_{1,2}(\mu)} \right\rfloor + 1,
\end{align*}
reduce to
\begin{align}
\nu(\mu)\neq 2m\alpha(\mu), \qquad 2\leq 2m \leq \left\lfloor \frac{\nu(\mu)}{\alpha(\mu)} \right\rfloor + 1.
\label{eq:nonresonance_conditions_3D}
\end{align}
The following two lemmas show that under these very general assumptions, the nonresonance conditions \eqref{eq:nonresonance_conditions} are guaranteed to fail arbitrarily close to a Hopf bifurcation.

\begin{lemma}[Failure of uniform nonresonance near Hopf bifurcation]
\label{lem:failure_uniform_nonresonance}

Under the assumptions \ref{assump:bifurcating_eigenvalue_pair}-\ref{assump:simple_disjoint_eigvals}, the nonresonance conditions \eqref{eq:nonresonance_conditions_3D}
cannot hold uniformly on any open neighborhood of $\mu_0$.
\end{lemma}
\begin{proof}
See Appendix \ref{sec:proof_lemma_failure_uniform_nonresonance}.
\end{proof}
The assumption \ref{assump:simple_disjoint_eigvals} implies that the eigenvalues $\lambda_{1,2,3}(\mu)$ depend $C^r$ smoothly on $\mu$ (\authcite{sibuyaGlobalPropertiesMatrices1965}). This means that we can write the expansions $\alpha(\mu)=a(\mu_0-\mu)^p+o(|\mu_0-\mu|^p)$ for some constant $a<0$, some integer $1\leq p \leq r$ and
$\nu(\mu)=\nu_0+b(\mu_0-\mu)+o(|\mu_0-\mu|)$ for some constant $\nu_0<0$.
\begin{lemma}[Asymptotic localization of resonances near a Hopf bifurcation]
\label{lem:resonance_locations}

Under the assumptions \ref{assump:bifurcating_eigenvalue_pair}-\ref{assump:simple_disjoint_eigvals}, assume in addition that
$r>1$.
Then there exists a sequence of parameters $\{\mu_{2m}, m\in\mathbb{N}\}$ with
$\mu_{2m}<\mu_0$ and $\mu_{2m}\to\mu_0$ as $m\to\infty$ such that the
nonresonance condition \eqref{eq:nonresonance_conditions_3D}
fails at $\mu=\mu_{2m}$.
Moreover, the resonance locations satisfy the asymptotic expansion
\begin{align*}
    \mu_{2m}=\mu_0-\left(\frac{\nu_0}{2am}\right)^{1/p}+o(m^{-1/p}),
    \qquad m\to\infty.
\end{align*}
\end{lemma}
\begin{proof}
See Appendix \ref{sec:proof_lemma_resonance_locations}.
\end{proof}
An important consequence of the localization of these resonance locations $\mu_{2m}$, is that for all $\mu>\mu_{2m}$, there exist no resonances between $\lambda_{1,2}(\mu)$ and $\lambda_{3}(\mu)$ of order $2m$ or lower. If there exist multiple purely real eigenvalues, it is sufficient to consider the one with the largest real part in order to find this lower bound. Furthermore, the arguments we have made to localize and classify resonances are not just limited to Hopf bifurcations; for any local bifurcation where a real eigenvalue crosses the imaginary axis and another negative real eigenvalue remains outside the selected spectral subspace, one generically encounters an accumulation of resonances of increasing order (this time involving all integer multiples instead of even multiples), leading to the same failure of uniform nonresonance with respect to the parameter $\mu$. Indeed, by the same argument, this phenomenon will also occur in bifurcations of higher codimension.

\subsubsection{Existence of invariant manifolds at resonance}
As established in Section \ref{sec:SSM_Theory}, away from resonances, there exists a $C^r$ smooth, unique, invariant SSM $W(\mu)$ tangent to $E^{\text{slow}}$. Due to the accumulation of resonances upon approaching bifurcation, we will however always encounter such resonances and as such have to address the existence of invariant manifolds at resonances by other means. For this, let us assume a spectral gap 
\begin{align*}
    \nu(\mu) < k\alpha(\mu),
\end{align*}
for a $k \in \N$ and a fixed $\mu$.
Because of \eqref{eq:ordering_eigenvalues}, there exists a spectral splitting for $\gamma(\mu) \in (\nu(\mu), \alpha(\mu))$, that reads
\begin{align*}
    \R^N = E_\gamma(\mu)^- \oplus E_\gamma(\mu)^+, \qquad E_\gamma(\mu)^- = \bigoplus_{\Re\lambda(\mu) < \gamma(\mu)} E_\lambda(\mu), \qquad E_\gamma(\mu)^+ = \bigoplus_{\Re\lambda(\mu) > \gamma(\mu)} E_\lambda(\mu) = E^s(\mu).
\end{align*}
Based on \authcite{irwinNewProofPseudostable1980}, \authcite{delallaveIrwinsProofPseudostable1997}, and \authcite{chenInvariantFoliationsC1Semigroups1997}, there exists a (generally non-unique) invariant pseudo-unstable manifold, $W_\gamma^+(\mu)$, tangent to $E^s(\mu)$ that is $C^k$ smooth. Importantly, this result does not assume nonresonance. In fact, at a
resonance of order $r_0$, i.e.
\begin{align*}
    \nu(\mu) = r_0\alpha(\mu). 
\end{align*}
we have the spectral gap \begin{align*}
    \nu(\mu) < (r_0-1)\alpha(\mu),
\end{align*}
meaning $W_\gamma^+(\mu)$ is $C^{r_0 -1}$ smooth. As we approach the bifurcation, the spectral gap will increase and $W_\gamma^+(\mu)$ will become smoother.

\begin{remark}
    For the proofs of Lemmas \ref{lem:failure_uniform_nonresonance} and \ref{lem:resonance_locations} we don't require the eigenvalue ordering assumption in \ref{assump:bifurcating_eigenvalue_pair} - as soon as there is a negative real eigenvalue the resonances will occur. Without assuming this however, the spectral gap discussed in this section will be smaller which leads to less smooth pseudo-invariant manifolds.
\end{remark}

\subsubsection{Summary of invariant manifolds near a Hopf Bifurcation}
\label{sec:summary_of_invariant_manifolds_near_a_hopf_bifurcation}
As pointed out in Section \ref{sec:accumulation_of_resonances_near_Hopf}, the nonresonance conditions required for the existence of SSMs will generally fail on a sequence of parameters $\{\mu_{2m}, m\in\mathbb{N}\}$ with  $\mu_{2m}<\mu_0$ that accumulates on $\mu_0$. However, as $\mu$ approaches $\mu_0$, the order of those resonances will increase. Because of this, let $\Tilde{V}_{2m}$ be an open interval with $\mu_0 \in \Tilde{V}_{2m}$ on which there exist resonances only of order higher than $2m$, i.e.
\begin{align*}
    \Tilde{V}_{2m} = \{\mu \in (\mu_-, \mu_+): m_1\lambda_1(\mu) + m_2 \lambda_2(\mu) = \lambda_l(\mu)\implies m_1 + m_2 > 2m\}
\end{align*}
Then, according to Lemma \ref{lem:resonance_locations}, we can choose 
\begin{align}
    \Tilde{V}_{2m} = (\mu_{2m}, \mu_0+M),
    \label{eq:nonresonance_assumption}
\end{align}
for some fixed $M>0$ for which \eqref{eq:ordering_eigenvalues} continues to hold. The lowest resonance we can encounter on $\Tilde{V}_{2m}$ is of order $2m+2$. Hence, at each parameter value $\mu \in \Tilde{V}_{2m}$ we are guaranteed
\begin{itemize}
    \item a $C^r$ unique primary SSM for all $\mu < \mu_0$ and away from resonance,
    \item a $C^{2m+1}$ pseudo-unstable manifold for all $\mu < \mu_0$ at resonance,
    \item a $C^r$ center manifold at $\mu = \mu_0$ which is unique for Hopf bifurcations,
    \item a $C^r$ unique unstable manifold for all $\mu > \mu_0$.
\end{itemize}
Therefore, we can safely calculate a Taylor expansion for an invariant manifold up to order $2m+1$ at each $\mu \in \Tilde{V}_{2m}$.

\subsection{Persistence of low-order SSM coefficients through Hopf bifurcation}
\label{sec:persistence_ssm_coefficients}
In this section we show that, although the existence of class $C^r$ SSMs breaks down on a sequence of bifurcation parameter values near a Hopf bifurcation due to accumulating resonances, their low-order Taylor coefficients are uniquely computable and persist smoothly across the bifurcation.

Let $P_E(\mu)$ be the spectral projection onto the real spectral subspace $E(\mu)$ associated to $\lambda_{1,2}(\mu)$. Let $\SS_{\uu}(\mu)\in \R^{N\times 2}$ be a basis of $E(\mu)$ and $\uu$ be the reduced coordinate defined by $\SS_{\uu}(\mu)\uu = P_E(\mu) \xx$.
The Taylor expansion for the SSM, $W(E;\mu)$, up to order $2m+1$ as a graph over $E(\mu)$ reads
\begin{align}
    \xx = \WW(\uu;\mu) = \SS_{\uu}(\mu)\uu + \sum_{2\leq\abs{\kk}\leq 2m+1} \WW^{\kk}(\mu) \uu^{\kk} + o(\norm{\uu}^{2m+1}),
    \label{eq:ssm_expansion}
\end{align}
where $\kk = (k_1,k_2)$ is a multi-index, $|\kk| = k_1 + k_2$, and $\uu^\kk = u_1^{k_1} u_2^{k_2}$. 
Next, we show that these Taylor coefficients $\WW^{\kk}(\mu)$ of the SSM up to order $2m+1$ are smooth over $\mu \in \Tilde{V}_{2m}$.
\begin{lemma}[Persistence of low-order SSM coefficients]
\label{lem:low_order_ssm_coeffs}

Under the assumptions \ref{assump:bifurcating_eigenvalue_pair}-\ref{assump:simple_disjoint_eigvals}, 
the SSM coefficients $\WW^{\kk}(\mu)$ are uniquely computable
for all multiindices $|\kk|\le 2m+1$ and all $\mu\in\tilde V_{2m}$ as defined in \eqref{eq:nonresonance_assumption}.
Moreover, the coefficients $\WW^\kk(\mu)$ vary $C^{r}$ smoothly in $\mu$.
\end{lemma}
\begin{proof}
    See Appendix \ref{sec:proof_lemma_low_order_ssm_coeffs}
\end{proof}
The full dynamics restricted to $W(E;\mu)$ are conjugate to the reduced dynamics
\begin{align}
    \dot{\uu}=\RR(\uu, \mu) =\sum_{1 \leq |k|\leq 2m+2} \RR^{\kk}(\mu) \uu^{\kk}+ o(\norm{\uu}^{2m+2}),
    \label{eq:red_dyn_expansion}
\end{align}
where we have written the Taylor expansion up to order $2m+2$. The following lemma shows that these coefficients also smoothly persist across the resonance.
\begin{lemma}[Persistence of low-order reduced dynamics coefficients]
\label{lem:low_order_red_dyn_coeffs}

Under the assumptions \ref{assump:bifurcating_eigenvalue_pair}-\ref{assump:simple_disjoint_eigvals}, 
the reduced dynamics coefficients $\RR^{\kk}(\mu)$ are uniquely computable
for all multiindices $|\kk|\le 2m+2$ and all $\mu\in\tilde V_{2m}$ as defined in \eqref{eq:nonresonance_assumption}.
Moreover, the coefficients $\RR^\kk(\mu)$ vary $C^{r}$ smoothly in $\mu$.
\end{lemma}
\begin{proof}
    See Appendix \ref{sec:proof_lemma_low_order_red_dym_coeffs}
\end{proof}
We summarize these findings in the following theorem.
\begin{theorem}[Smooth computation of low-order SSM coefficients through a Hopf bifurcation]
\label{thm:low_order_ssm_hopf}

Consider the system \eqref{eq:system_around_fixed_point} with a Hopf bifurcation at $\mu = \mu_0$, satisfying the spectral condition \ref{assump:bifurcating_eigenvalue_pair}. Let $\Tilde{V}_{2m} = (\mu_{2m}, \mu_0+M)$ for some fixed $M>0$ for which \ref{assump:bifurcating_eigenvalue_pair}-\ref{assump:simple_disjoint_eigvals} hold, with $\mu_{2m}$ defined as in Lemma \ref{lem:resonance_locations}.
Then the following statements are true:

\begin{enumerate}
    \item For all $\mu \in \Tilde{V}_{2m}$, for which the nonresonance conditions \eqref{eq:nonresonance_conditions} hold, there exists a unique, class $C^r$, invariant, 2D slow SSM, $W(\mu)$, tangent to the slow eigenspace $E^{\mathrm{slow}}(\mu)$ at $\xx_0(\mu)$. At resonances, $W(\mu)$ is $C^{2m+1}$ and non-unique.

    \item For $\mu > \mu_0$, $W(\mu)$ coincides with the 2D unstable manifold $W^u(\mu)$, and for $\mu \downarrow \mu_0$, it continues smoothly to the center manifold $W^c(\mu_0)$.

    \item The Taylor coefficients of $W(\mu)$ up to order $2m+1$,
    \begin{align*}
        \xx = \WW(\uu;\mu) = \SS_{\uu}(\mu)\uu + \sum_{2\leq\abs{\kk}\leq 2m+1} \WW^{\kk}(\mu) \uu^{\kk} + o(\norm{\uu}^{2m+1}),
    \end{align*}
    are $C^{r}$ smooth in $\mu$ across $\Tilde{V}_{2m}$.

    \item The Taylor coefficients of the reduced dynamics up to order $2m+2$,
    \begin{align*}
        \dot{\uu}=\RR(\uu, \mu) =\sum_{1 \leq |k|\leq 2m+2} \RR^{\kk}(\mu) \uu^{\kk}+ o(\norm{\uu}^{2m+2}),
    \end{align*}
    are $C^{r}$ smooth in $\mu$ across $\Tilde{V}_{2m}$.
\end{enumerate}
\end{theorem}
Analogous results hold for other types of bifurcations too: beyond a parameter value after which all remaining resonances are of high order, all low-order Taylor coefficients will persist smoothly. The uniqueness of the center manifold is specific to Hopf bifurcations.

\section{Results}

In this section we demonstrate our theoretical findings on two examples. The first is a minimal example, which is general enough to show that the loss of smoothness at resonance is a generic phenomenon near bifurcations. 
Our second example is a high-dimensional finite-element simulation of the Navier-Stokes equations, for which we construct two-dimensional parametric reduced-order model from data that accurately predicts the full transition to periodic dynamics. 
Moreover, this data-driven approach is able to overcome the smoothness limitations that occur due to resonances in the equation-driven setting. 

\subsection{Parametric SSMs for the Hopf Normal Form}

We consider the classic cubic Hopf normal form (see, e.g., \authcite{guckenheimerNonlinearOscillationsDynamical1984}) with an additional stable transverse direction as a canonical example for the construction and regularity analysis of SSMs near a Hopf bifurcation. 
\begin{align}
    \dot{x} &= \mu x - \omega y - x(x^2+y^2), \nonumber\\
    \dot{y} &= \omega x + \mu y - y(x^2+y^2), \label{eq:toy_model_hopf} \\
    \dot{z} &= \sigma z + x^2+y^2, \qquad \sigma < 0. \nonumber
\end{align}
The subsystem in $(x,y)$ captures the leading-order dynamics near any nondegenerate Hopf bifurcation, while the additional transverse $z$-direction represents the simplest nonlinear coupling relevant for the existence and smoothness of SSMs. For this reason, pathologies in the regularity of the SSM expansion present in this system will also typically occur in high-dimensional systems undergoing a Hopf bifurcation.
Similar systems have been examined in the past as standard examples to show the loss of regularity for center manifolds, (see, e.g., \authcite{guckenheimerNonlinearOscillationsDynamical1984}, \authcite{vanderbauwhedeCentreManifoldsNormal1989}, and \authcite{vanstrienCenterManifoldsAre1979}).

In complex coordinates $u=x+iy$, we can write system \eqref{eq:toy_model_hopf} as
\begin{align*}
    \dot{u} &= F(u) = (\mu+\omega i)u - \abs{u}^2u, \\
    \dot{z} &= \sigma z + \abs{u}^2.
\end{align*}

Due to the rotational symmetry $F(e^{i\theta}u) = e^{i\theta}F(u)$, any invariant manifold over $(u,\Bar{u})$ also has to be rotationally invariant. Note that a function $h$ that is real analytic on an open domain including the origin, is rotationally invariant if and only if $h(u,\Bar{u})=h(\abs{u}^2)$. Indeed, expanding 
\begin{align*}
    h(u,\Bar{u}) = \sum_{n,m=1}^\infty a_{nm} u^n \Bar{u}^m,
\end{align*}
and enforcing
\begin{align*}
    (e^{i\theta}u)^n (\overline{e^{i\theta}u})^m &= e^{(n-m)i\theta}u^n \Bar{u}^m \equiv u^n \Bar{u}^m,
\end{align*}
implies $n=m$, and hence
\begin{align*}
    u^n \Bar{u}^n &= \abs{u}^{2n}.
\end{align*}
More generally, this property follows from $SO(2)$ invariance of the system (see, e.g., \authcite{fieldDynamicsSymmetry2007}).

Therefore, we can seek an expression, $z=h(\rho)$, where $\rho=\abs{u}^2$ and
\begin{align*}
    \dot{\rho} &= 2\mu \rho - 2\rho^2, \\
    \dot{z} &= \sigma z + \rho.
\end{align*}
Substituting the formal Taylor expansion
\begin{align*}
    z = h(\rho) = \sum_{k=1}^\infty a_k \rho^k,
\end{align*}
into the invariance equation
\begin{align}
    \sigma h(\rho)+\rho = 2h'(\rho)(\mu \rho-\rho^2),
    \label{eq:invariance_ode_toy}
\end{align}
then shifting the indices yields
\begin{align}
    \sum_{k=1}^{\infty} \sigma a_k \rho^k + \rho =  \sum_{k=1}^{\infty} 2k\mu a_k \rho^k - \sum_{k=2}^{\infty} 2(k-1) a_{k-1} \rho^k.
    \label{eq:expansion_in_invariance_ode_toy}
\end{align}
Equating equal orders of $\rho$ in \eqref{eq:expansion_in_invariance_ode_toy} we obtain the recursion
\begin{align}
    a_1 &= \frac{1}{2\mu - \sigma}, \nonumber \\
    a_{k} &= \frac{2(k-1)}{2k\mu - \sigma} a_{k-1}, \qquad k\geq 2.
    \label{eq:recursion_coeffs}
\end{align}
The convergence radius of this power series can be computed as 
\begin{align*}
    \rho_0 = \lim_{k\rightarrow \infty}\abs{\frac{a_k}{a_{k+1}}}=\abs{\mu}.
\end{align*}
In $(x,y)$-coordinates, this corresponds to a radius of $r_0=\sqrt{\abs{\mu}}$, which also matches the radius of the limit cycle arising in the Hopf bifurcation.
The reduced dynamics, $\dot{\rho} = 2\mu \rho - 2\rho^2$, on the parametric SSM remain smooth for every $\mu$.

Next, we will characterize solutions to \eqref{eq:invariance_ode_toy} more precisely and assess their regularity. As the loss of regularity only happens pre-bifurctation, we limit our analysis to $\mu<0$. For this, notice that the equation \eqref{eq:invariance_ode_toy} can be solved on $\rho\in (0,\infty)$ using an integrating factor approach, yielding 
\begin{align}
    h(\rho)&=M(\rho)^{-1}M(\rho_0)h(\rho_0)-M(\rho)^{-1}\int_{\rho_0}^{\rho}\frac{M(s)}{2(s-\mu)}ds = h_h(\rho)+h_p(\rho),
    \label{eq:true_solution_toy_example}\\
    M(\rho)&=\left(\frac{\abs{\mu}-\sign(\mu)\rho}{\rho}\right)^{\sigma/(2\mu)} \nonumber,
\end{align}
where the first summand in \eqref{eq:true_solution_toy_example} is the homogeneous solution and the second is the particular solution. We let $\alpha = \sigma/(2\mu)$ and write the integrand in \eqref{eq:true_solution_toy_example} as
\begin{align*}
    \frac{M(s)}{2(s-\mu)} = \frac{1}{2s}\left(1+\frac{\abs{\mu}}{s}\right)^{\alpha-1} = \frac{1}{2s}\sum_{k=0}^{\infty} \binom{\alpha-1}{k}  \left(\frac{\abs{\mu}}{s}\right)^{\alpha-1-k},
\end{align*}
where we used the generalized binomial theorem which is uniformly convergent for $s<\abs{\mu}$. Therefore, we can swap summation and integration over the compact interval $[\rho_0, \rho]$ in order to integrate each summand individually. This gives us the final result
\begin{align*}
    h_p(\rho) = -\frac{1}{2} \left( \frac{\rho}{\rho+\abs{\mu}} \right)^{\alpha}\Bigg(&\sum_{\substack{k=0, \\ k\neq \alpha-1}}^{\infty}\binom{\alpha-1}{k} \frac{\abs{\mu}^{\alpha-1-k}}{k-\alpha+1}\left(\rho^{k-\alpha+1} - \rho_0^{k-\alpha+1}\right) \\
    &+ \log\left(\frac{\rho}{\rho_0}\right)
    \mathbbm{1}_{\Z^{+}}(\alpha)
    \Bigg),
\end{align*}
which satisfies $\lim_{\rho \downarrow 0}h_p(\rho)=0$. Here, $\mathbbm{1}_A$ denotes the indicator function of a set $A$.

For $\alpha \notin \mathbb{Z}^+$, the homogeneous solution in \eqref{eq:true_solution_toy_example} can be chosen so that $h$ is analytic at the origin, and the corresponding invariant manifold in $(x,y)$ is also analytic. The other solutions (called fractional SSMs, see  \authcite{hallerNonlinearModelReduction2023}), are in the the Hölder space $C^{\lfloor \alpha \rfloor, \alpha - \lfloor \alpha \rfloor}$ around the origin, which in $(x,y)$-coordinates corresponds to $C^{\lfloor 2\alpha \rfloor, 2\alpha - \lfloor 2\alpha \rfloor}$.

When $\alpha = k_0 \in \mathbb{Z}^+$, resonance leads to terms like $\rho^{k_0}\log(\rho)$, so the smoothness is $C^{k_0-1, \delta}$ for any $\delta<1$ in $\rho$ and $C^{2k_0-1, \delta}$ in $(x,y)$. This loss of regularity matches the singularity in the Taylor coefficient $a_{k_0}$ at resonance and shows that the smoothness guarantees at resonance in Section \ref{sec:summary_of_invariant_manifolds_near_a_hopf_bifurcation} are sharp, because $\alpha$ determines the spectral gap.

In summary, for $\alpha = \sigma/(2\mu)$, all invariant manifolds in $(x,y)$ belong to
a smoothness class depending on $\alpha$ at the origin, with the existence of a unique analytic solution guaranteed if and only if $\alpha \notin \mathbb{Z}^+$. Note that this is consistent with Hartman's theorem (\authcite{hartmanLocalHomeomorphismsEuclidean1960}), as the lowest possible maximal smoothness is $C^1$, which occurs at $\alpha=1$.

For $\mu > 0$, \eqref{eq:true_solution_toy_example} remains the true general solution of the invariance equation \eqref{eq:invariance_ode_toy} for $\rho \in (0,\mu)$. With the approach of expanding the integrand for $s<\mu$, we calculate
\begin{align*}
    h_p(\rho) &= \frac{1}{2} \left( \frac{\rho}{\mu-\rho} \right)^{\alpha}\Bigg(\sum_{k=0}^{\infty}\binom{\alpha-1}{k} \frac{(-1)^k\mu^{\alpha-1-k}}{k-\alpha+1}\left(\rho^{k-\alpha+1} - \rho_0^{k-\alpha+1}\right)\Bigg).
\end{align*}
Therefore, in contrast to the pre-bifurcation case, no loss of regularity occurs for $\mu>0$. The analytic solution coincides with the unstable manifold of the system, and it exists for every $\mu>0$ as guaranteed by the center manifold theorem. Here, the analytic manifold also uniquely satisfies $h(0)=0$, as all other solutions blow up at the origin.
\begin{figure}[ht]
    \centering
    \begin{subfigure}[t]{0.45\linewidth}
        \includegraphics[width=\linewidth]{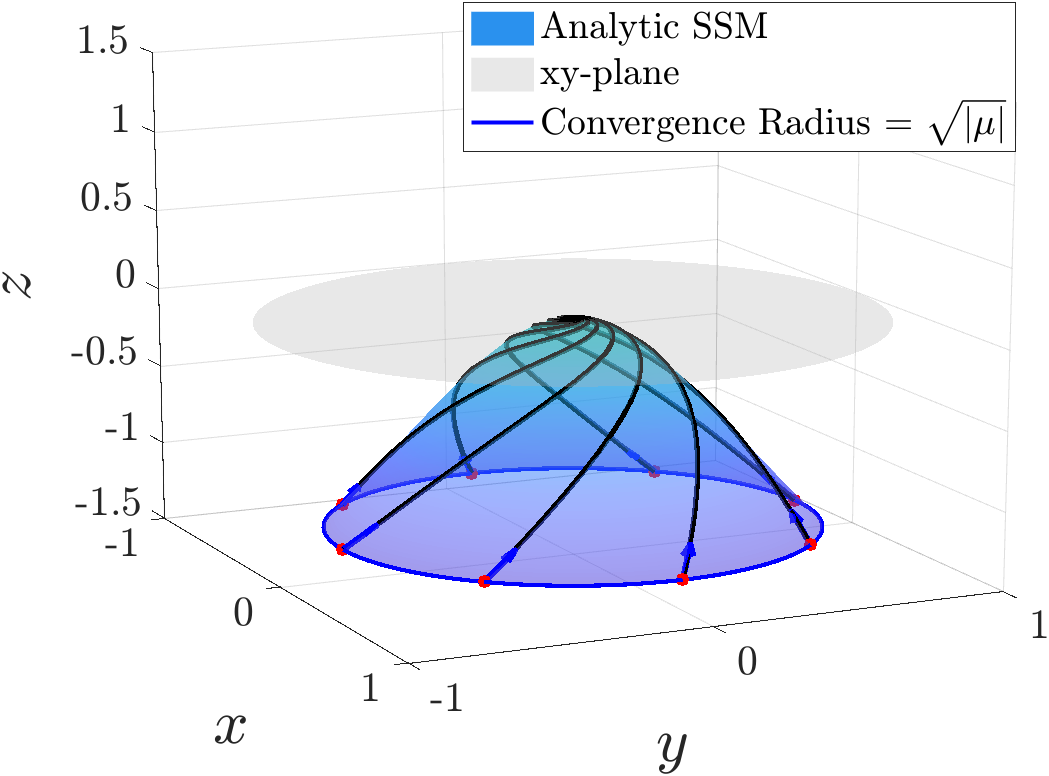}
        \caption{Analytic SSM before the $\alpha=1$ resonance.}
        \label{fig:sub1}
    \end{subfigure}
    \hfill
    \begin{subfigure}[t]{0.45\linewidth}
        \includegraphics[width=\linewidth]{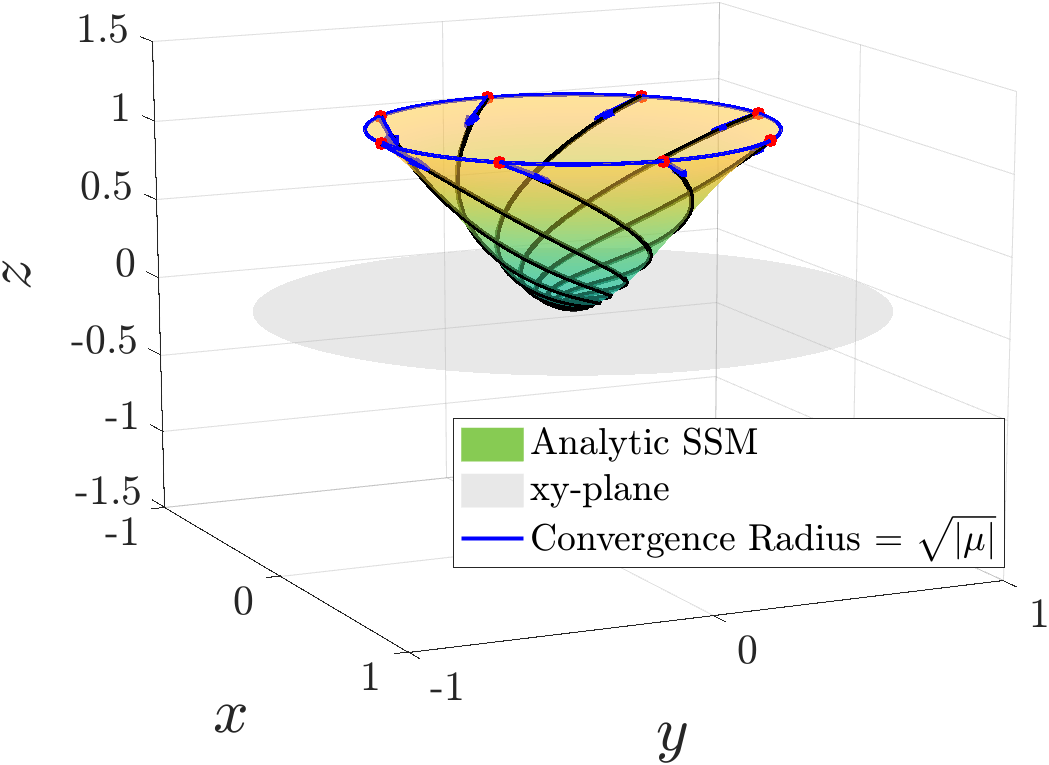}
        \caption{Analytic SSM after the $\alpha=1$ resonance.}
        \label{fig:sub2}
    \end{subfigure}
    \hfill
    \begin{subfigure}[t]{0.45\linewidth}
        \includegraphics[width=\linewidth]{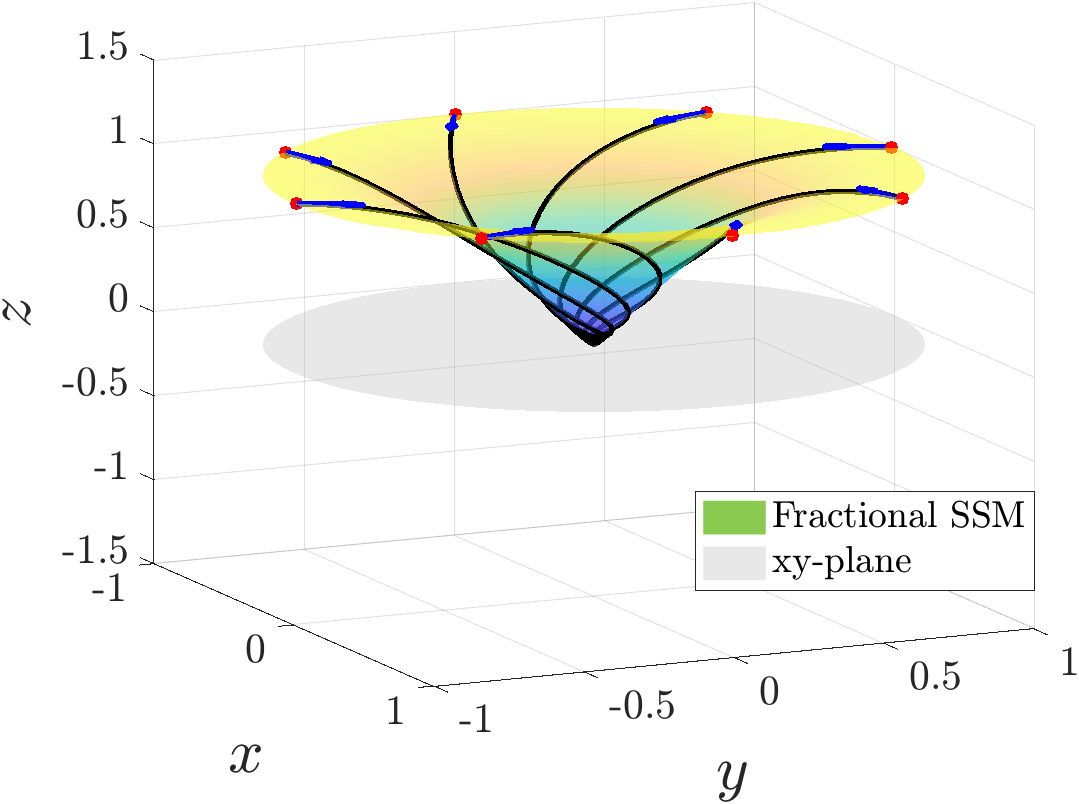}
        \caption{Fractional SSM persisting through resonance.}
        \label{fig:sub3}
    \end{subfigure}
    \hfill
    \begin{subfigure}[t]{0.45\linewidth}
        \includegraphics[width=\linewidth]{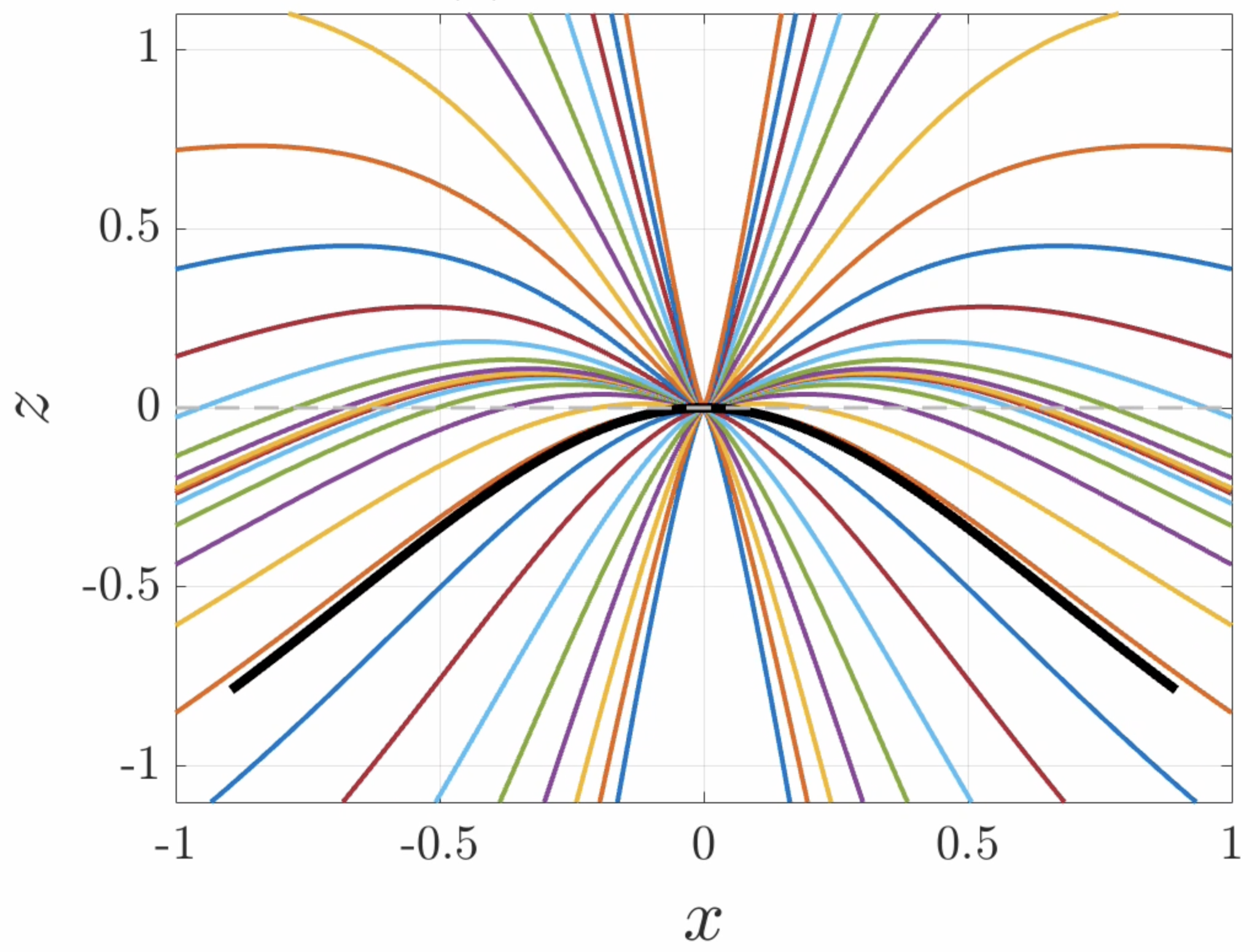}
        \caption{The $y=0$ cross-section of fractional SSMs for different $h(\rho_0)$ (coloured) and analytic SSM (black).}
        \label{fig:sub4}
    \end{subfigure}
    \caption{Analytic and fractional SSMs of system \eqref{eq:toy_model_hopf} with trajectories overlayed in (a)-(c), and as a cross-section in (d).}
    \label{fig:all_four}
\end{figure}

\begin{remark}
As the transformation to normal form has a finite domain of convergence, it is worth considering the case in which the subsystem of \eqref{eq:toy_model_hopf} in $(x,y)$ contains higher-order nonlinearities or $z$-dependent terms. If rotational symmetry is preserved, then inserting the formal Taylor expansion and equating orders leads to a recursion similar to \eqref{eq:recursion_coeffs}, but the coefficient $a_k$ will depend recursively on multiple earlier coefficients. The resonances, however, appear in the same manner.

If instead the order of nonlinearity in the $z$-equation is $2p > 2$ for $p \in \N$, (i.e., the leading-order nonlinearity might be missing), then the resonances introduce nonunique solutions for low-order coefficients and there exists a solution whose Taylor coefficients only become singular after order $k=p$.
\end{remark}

\subsection{Data-Driven Parametric SSMs for the Lid-Driven Cavity Flow}
\label{sec:lid_driven_cavity}

Originally explored by \authcite{kawagutiNumericalSolutionNavierStokes1961} in the early 1960s, the lid-driven cavity has become a standard benchmark for incompressible Navier-Stokes solvers. To this day, the results of \authcite{ghiaHighReSolutionsIncompressible1982} constitute the most commonly used comparison. For a review of the relevant literature, we refer to \authcite{kuhlmannLidDrivenCavity2018}.

This system gives rise to a transition from laminar to chaotic flow depending on the Reynolds number $\mu$. As the latter is increased, the system undergoes a supercritical Hopf bifurcation, followed by quasiperiodic motion as subsequent pairs of eigenvalues cross the imaginary axis.
Multiple studies have been undertaken to localize the critical Reynolds value $\mu_{\text{crit}}$ at which the Hopf bifurcation takes place. While outliers do exist in the literature, the general consensus agrees with the bounds $\mu_{\text{crit}}\in [8000, 8050]$ reported by \authcite{bruneau2DLiddrivenCavity2006}. Notable studies falling within this range include \authcite{fortinLocalizationHopfBifurcations1997} with $\mu_{\text{crit}} = 8000$ and, more recently, \authcite{auteriNumericalInvestigationStability2002} with $\mu_{\text{crit}}\in [8017.6, 8018.8]$. By solving the generalized eigenproblem associated to the linearized Navier-Stokes equation (see Appendix \ref{sec:discretization_navier_stokes}), we find the critical value to be $\mu_{\text{crit}} \approx 8015$, which is in good agreement with the literature. A better agreement has been obtained for the critical oscillation frequency $\omega_{\text{crit}}\in [2.83, 2.87]$. Indeed, in this study we find $\omega_{\text{crit}} = 2.83$ to be in agreement with this bound.

\subsubsection{Setup}

The dynamics of the lid-driven cavity are governed by the incompressible Navier-Stokes equations on the unit square $\Omega = (0,1)^2\subseteq \R^2$, given by
\begin{align}
    \partial_t u - \nu \Delta u + (u \cdot \nabla)u + \nabla p &= 0, \nonumber \\
    \nabla \cdot u &= 0, \nonumber \\
    \Gamma u &= b,
    \label{eq:Navier_Stokes}
\end{align}
with velocity $u$, pressure $p$, kinematic viscosity $\nu$, trace operator $\Gamma$ onto the boundary $\partial\Omega$, and function $b$ on the boundary prescribing a velocity profile.
For the lid-driven cavity the standard choice is 
\begin{align*}
    b(\xx) = 
    \begin{cases}
        (1,0) & \text{on } [0,1]\times\{1\}, \\
        (0,0) & \text{on } \partial\Omega \setminus ([0,1]\times\{1\}),
    \end{cases}
\end{align*}
which amounts to no-slip conditions on the cavity walls and a prescribed unit velocity in the $x$-direction at the lid.
The central parameter in this problem is the Reynolds number $\mu= UL/\nu$, where $U=1$ is the lid velocity and $L=1$ is the cavity length.

We simulate the flow using the open source toolbox \textit{FlowControl}\footnote{https://github.com/williamjussiau/flowcontrol}, which uses the Finite Element Method with standard P2-P1 Taylor-Hood elements and a semi-implicit multistep method for time integration.
As shown in Appendix \ref{sec:discretization_navier_stokes}, after discretizing \eqref{eq:Navier_Stokes} in space, the evolution of the velocities can be viewed as a trajectory of a finite-dimensional parametric dynamical system,
\begin{align}
    \dot{\xx} = \ff(\xx;\mu), \qquad \xx\in \R^{N}, 
    \qquad \mu \in \R^+,
    \qquad \ff\in C^{\infty}(\R^{N}\times \R^+,\R^{N}).
    \label{eq:parametric_ode}
\end{align}
The reason we focus on the velocities is because of a pressure singularity that is known to form in the numerical treatment of the lid-driven cavity (\authcite{kuhlmannLidDrivenCavity2018}).

\subsubsection{Parametric SSM construction}
Using the eigenvalues of the linear part of the discretized Navier-Stokes equations, we find a parameter regime $\mu \in [7900,8500]$, in which the bifurcating eigenvalue pair have the largest real part while the rest of the spectrum remains stable. 
More specifically, in this parameter range, $2\Re\lambda_{1,2}(\mu)>\lambda_{3}(\mu)$ holds. This means that, as shown in Section \ref{sec:persistence_ssm_coefficients}, we can choose $m=1$ in \eqref{eq:ssm_expansion} and \eqref{eq:red_dyn_expansion}.
The Hopf bifurcation happens around $\mu \approx 8015$ according to our linear analysis. Finally, a second bifurcation happens for some $\mu > 8500$. 
Following the \texttt{SSMLearn} algorithm of \authcite{cenedeseDatadrivenModelingPrediction2022}, we extract a 2D SSM from data for multiple values of the Reynolds number in order to interpolate their polynomial coefficients to obtain a parametric model.
Figure \ref{fig:parameter_range} shows the distribution of training and test parameters.
\begin{figure*}
    \centering
    \makebox[\textwidth][c]{%
        \includegraphics[width=1.5\textwidth]{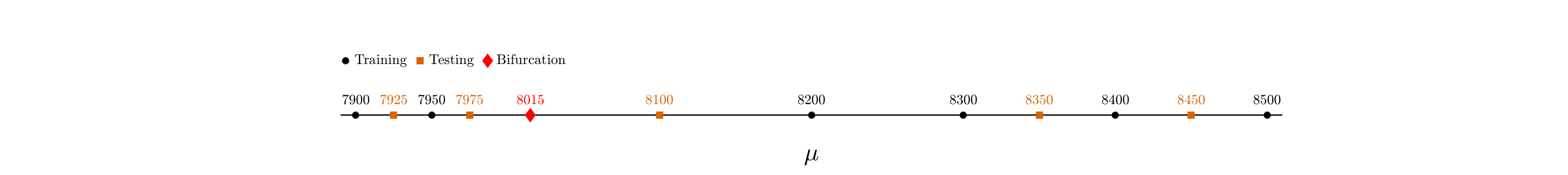}
    }
    \caption{Distribution of training and testing parameters.}
    \label{fig:parameter_range}
\end{figure*}

By generating data at each training parameter $\mu$, we construct a snapshot matrix whose dominant directions provide an approximation of the tangent space of the SSM. Since we seek a 2D SSM, retaining the two leading Proper Orthogonal Decomposition (POD) modes, which approximate the spectral subspace associated with the bifurcating eigenvalue pair, gives us a basis $\VV(\mu)\in \R^{N\times 2}$.
The chart yielding the reduced coordinates $\uu$ is then given by the orthogonal projection onto these POD modes,
\begin{align*}
    \uu = \VV(\mu)^T \xx.
\end{align*}
It is worth noting that singular value decompositions (SVDs) smoothly depend on parameters as long as blocks of singular values remain disjoint (see \authcite{dieciSmoothDecompositionsMatrices1999}). The SSM is then approximated as a multivariate polynomial graph of order $M_W$ over this approximate tangent space:
\begin{align*}
    \xx = \WW(\uu;\mu) = \VV(\mu) \uu + \sum_{2 \leq |\kk|\leq M_W} \WW^{\kk}(\mu) \uu^\kk.
\end{align*}
The polynomial coefficients $\WW_{\kk}$ are identified via least-squares regression by minimizing the reconstruction error (see Appendix \ref{sec:NMTE}) between the full-order state data $\xx$ and the corresponding reduced coordinates $\uu$.

Next, we seek the reduced dynamics, 
\begin{align*}
    \dot{\uu} = \RR(\uu;\mu) = \sum_{1 \leq |k|\leq M_R} \RR^{\kk}(\mu) \uu^{\kk},
\end{align*}
that are conjugate to the dynamics on the SSM, using a multivariate polynomial of order $M_R$. The time derivatives of the reduced coordinates $\dot{\uu}$ are obtained via finite differences, and the polynomial coefficients $\RR^{\kk}(\mu)$ of the reduced dynamics are again identified via least-squares regression.
To capture a Hopf bifurcation, we require at least $M_R\geq 3$.

As the coefficients are fit from data, it makes sense to choose $M_W$ and $M_R$ as low as possible to ensure robustness and avoid overfitting, while retaining high accuracy on unseen trajectories. According to Section \ref{sec:persistence_ssm_coefficients}, a choice of $M_W=2, \ M_R=3$ is justified given the spectrum. However, we find that ignoring the order $4$ resonance and choosing $M_W=4, \ M_R=5$ consistently yields better results. This is not inconsistent with the results in Section \ref{sec:persistence_ssm_coefficients}, as we are fitting polynomials to data instead of solving the invariance equations order by order using Taylor expansions.

\subsubsection{Data generation and Training}
Pre-Hopf bifurcation, we perturb the stable base flow moderately and truncate the data during its initial decay towards the SSM. Post-Hopf, the unstable base flow is perturbed by a very small amount. Due to the small real parts of the unstable eigenvalues, we choose this perturbation to be in the direction of the unstable eigenvectors, thus saving time with the numerical simulation. However, if the perturbation is chosen small enough and the simulation is run long enough, this specific choice of the perturbation direction is not necessary.

Altogether, this procedure yields data lying close to the SSM for different values of the Reynolds number. While we did use knowledge of the linear part in data generation, we emphasize that this is not required in a general setting.

For each training parameter, we fit a reduced order model using SSMLearn to a single trajectory. These models each achieve a low normalized mean trajectory error (NMTE) and normalized mean amplitude error (NMAE) on an unseen trajectory at that respective parameter value. For a definition of the error metrics, see Appendix \ref{sec:NMTE}.

\subsubsection{Interpolation}
As outlined in Section \ref{sec:SSM_Theory}, the fixed point $\xx(\mu)$, the spectral subspace $E(\mu)$ associated to $\lambda_{1,2}$, the low-order SSM coefficients $\WW^{\kk}(\mu)$, and the low-order reduced dynamics coefficients $\RR^{\kk}(\mu)$ all persist $C^\infty$ smoothly in $\mu$ across $\Tilde{V}$.
Therefore, we can justify interpolating these coefficients individually. The SSM models fitted on the training data are interpolated in $\mu$ using spline interpolation.

\subsubsection{Testing}
Interpolating the SSM models yields a parametric model for $\mu \in [7900,8500]$. We evaluate this model at unseen Reynolds numbers by providing an initial condition, projecting onto the interpolated tangent space, iterating the interpolated reduced dynamics and finally mapping to the SSM using the interpolated graph. 

By comparing to our ground truth simulation data from the same initial condition, we calculate the errors of this procedure averaged over $8$ test trajectories. These errors are displayed in Tables \ref{tab:NMTE_parametric_1} and \ref{tab:NMTE_parametric_2}. Our parametric model consistently achieves errors of $5\%$ or less.
The slightly higher error values at $\mu=8100$ can be explained by the fact that this testing value is the furthest away from any of the training parameters. This leads to a slight phase inaccuracy which can be seen by the phase-sensitive NMTE being higher than the NMAE. The Reconstruction Error remains low accross all testing parameters, which means that the trajectories do indeed lie close to the interpolated manifold. The choice of $M_W=4$ consistently outperforms the choice of $M_W=2$.

By calculating the eigenvalues of the linear part of the parametric reduced dynamics, we are able to make a prediction $\mu_{pred}$ for the bifurcation value. At $\mu_{pred} = 8019$, the prediction error to the true value that we calculated to be $\mu_0 = 8015$ is less than $0.05\%$. This value is remarkably close to that reported by \authcite{auteriNumericalInvestigationStability2002}.

\begin{table}[h!]
\centering
\begin{minipage}{0.48\linewidth}
\centering
\begin{tabular}{c c c c}
\hline
$\boldsymbol{\mu}$ & \multicolumn{3}{c}{\textbf{Errors (\%)}} \\
\cline{2-4}
 & NMTE & NMAE & RecError \\
\hline
7925 & 4.43 & 3.51 & 2.23 \\
7975 & 4.52 & 3.26 & 3.84 \\
8100 & 18.22 & 1.72 & 2.74  \\
8350 & 3.76 & 2.65 & 1.68 \\
8450 & 7.97 & 2.49 & 2.50 \\
\hline
\end{tabular}
\caption{Error metrics for $M_W=2, M_R=3$.}
\label{tab:NMTE_parametric_1}
\end{minipage}
\hfill
\begin{minipage}{0.48\linewidth}
\centering
\begin{tabular}{c c c c}
\hline
$\boldsymbol{\mu}$ & \multicolumn{3}{c}{\textbf{Errors (\%)}} \\
\cline{2-4}
 & NMTE & NMAE & RecError \\
\hline
7925 & 1.59 & 1.24 & 1.24 \\
7975 & 1.71 & 1.66 & 1.20 \\
8100 & 7.31 & 4.82 & 0.93  \\
8350 & 1.08 & 1.71 & 0.12 \\
8450 & 0.92 & 1.05 & 0.38 \\
\hline
\end{tabular}
\caption{Error metrics for $M_W=4, M_R=5$.}
\label{tab:NMTE_parametric_2}
\end{minipage}
\end{table}

\begin{figure}[htbp]
    \centering
    \begin{minipage}[t]{0.62\textwidth}
        \centering
        \begin{subfigure}[t]{\linewidth}
            \centering
            \includegraphics[width=\linewidth]{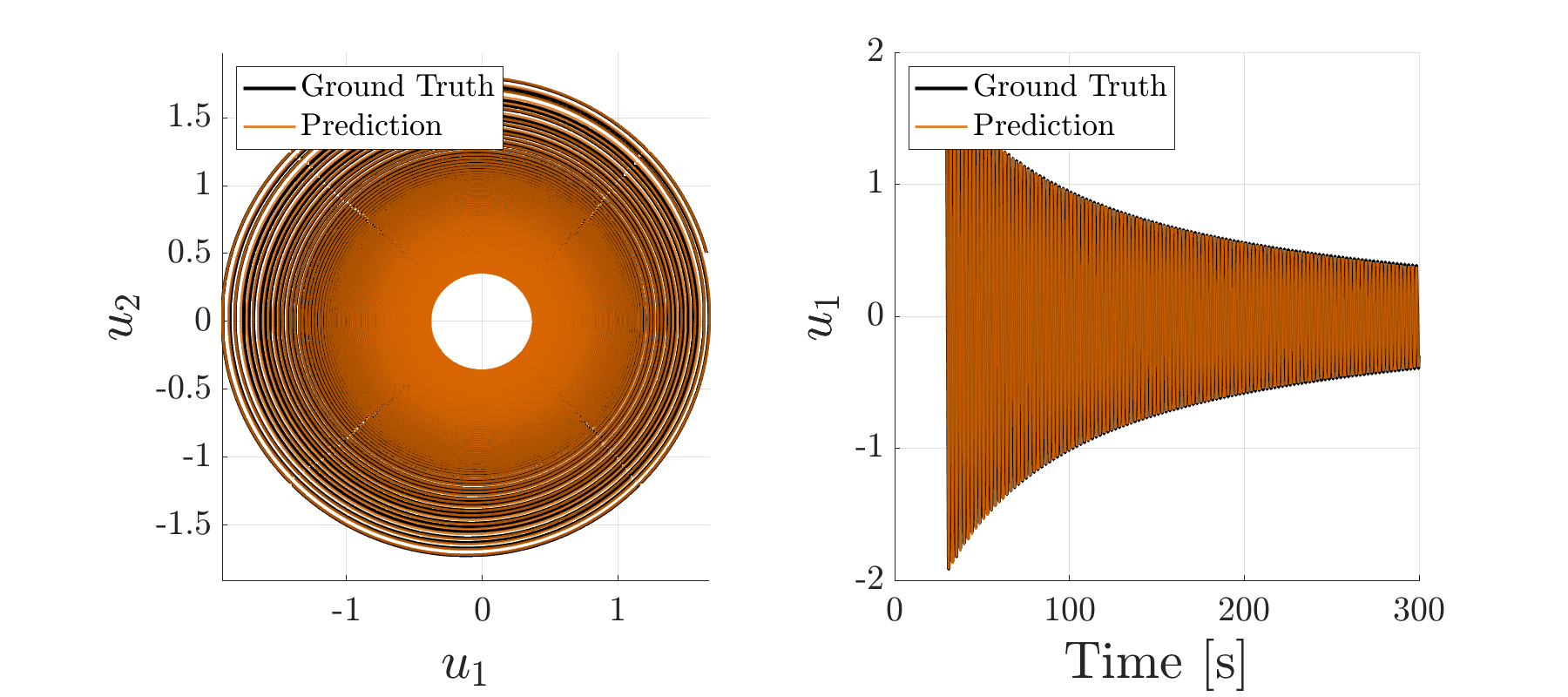}
            \caption{Reduced dynamics at $\mu = 7925$}
            \label{fig:reduced7925}
        \end{subfigure}
        \vspace{0.3em}
        \begin{subfigure}[t]{\linewidth}
            \centering
            \includegraphics[width=\linewidth]{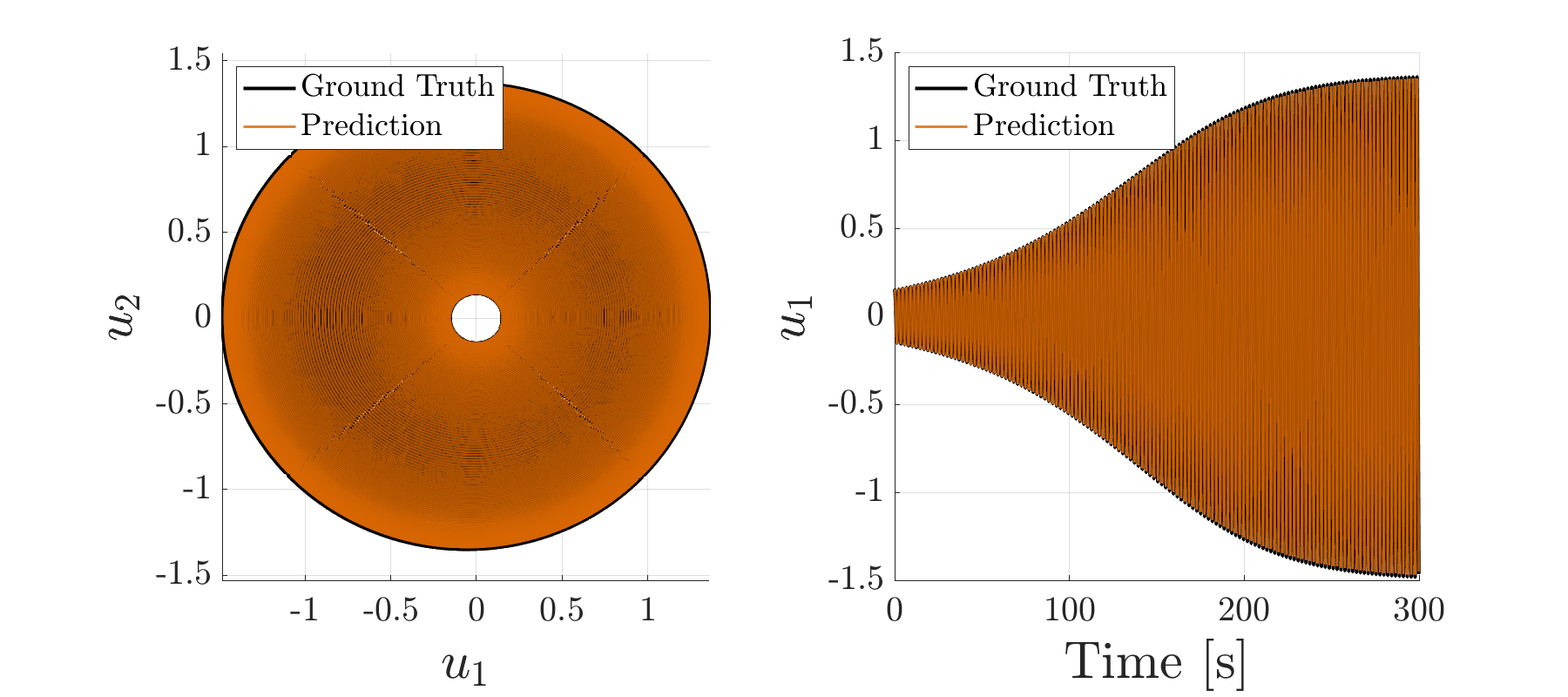}
            \caption{Reduced dynamics at $\mu = 8450$}
            \label{fig:reduced8450}
        \end{subfigure}
    \end{minipage}
    \hfill
    \begin{minipage}[c]{0.37\textwidth}
        \centering
        \vspace*{4em} 
        \begin{subfigure}[t]{\linewidth}
            \centering
            \includegraphics[width=\linewidth]{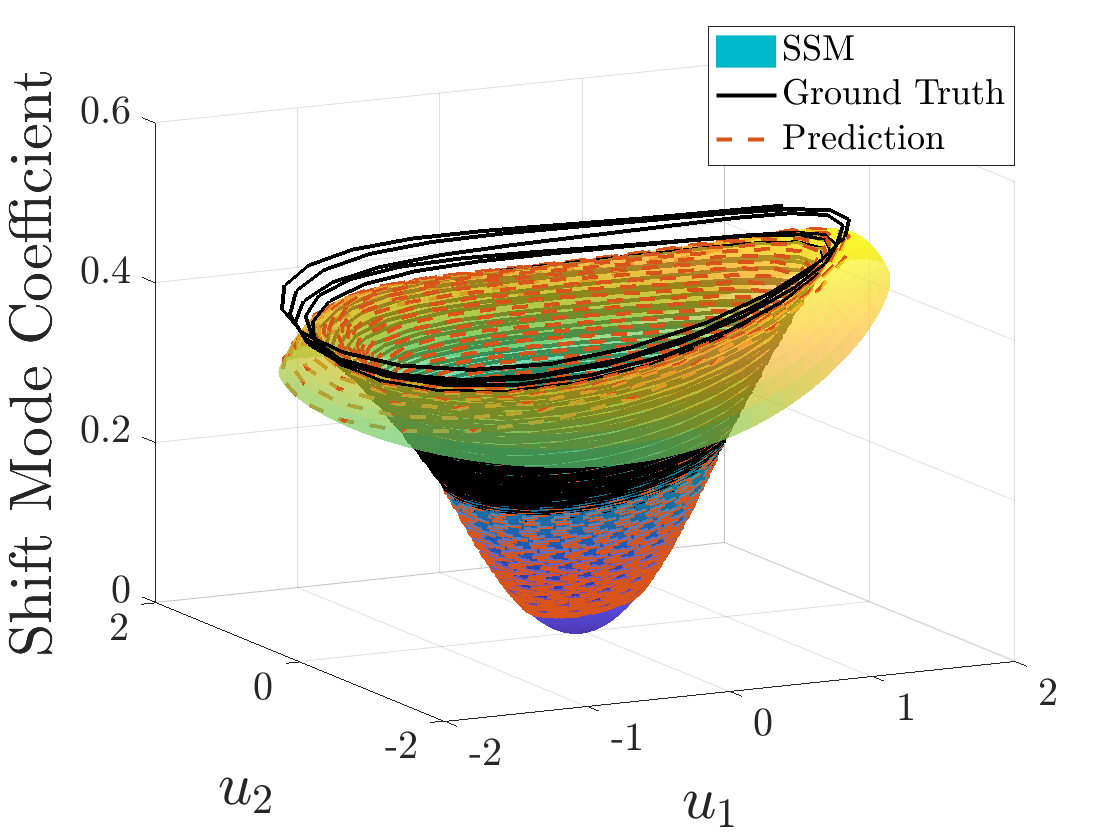}
            \caption{SSM at $\mu = 7925$}
            \label{fig:ssm7925}
        \end{subfigure}
        \vspace{0.5em}
        \begin{subfigure}[t]{\linewidth}
            \centering
            \includegraphics[width=\linewidth]{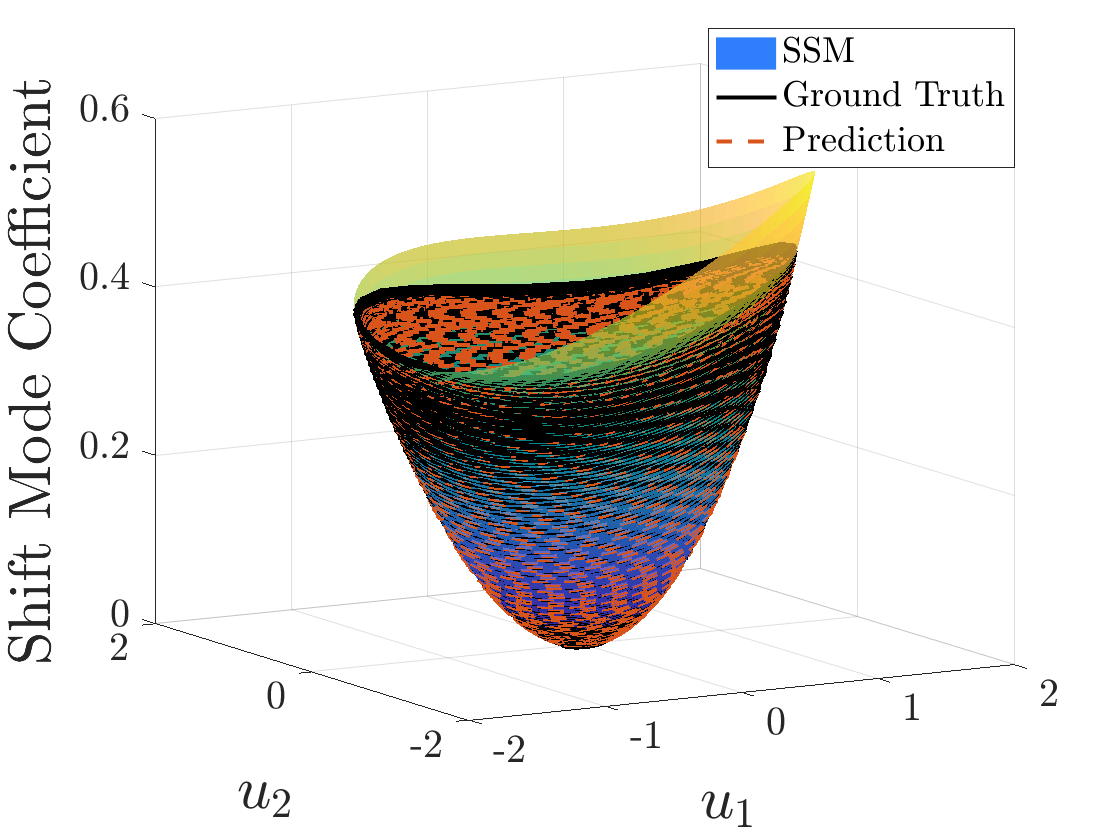}
            \caption{SSM at $\mu = 8450$}
            \label{fig:ssm8450}
        \end{subfigure}
    \end{minipage}
    \caption{Comparison of the SSM-reduced dynamics and SSMs at different previously unseen Reynolds numbers. The shift mode was introduced in \authcite{noackHierarchyLowdimensionalModels2003}}
    \label{fig:reduced_dynamics_layout}
\end{figure}

\begin{figure}[htbp]
    \centering
    \begin{subfigure}{0.4\textwidth}
        \centering
        \includegraphics[width=\linewidth]{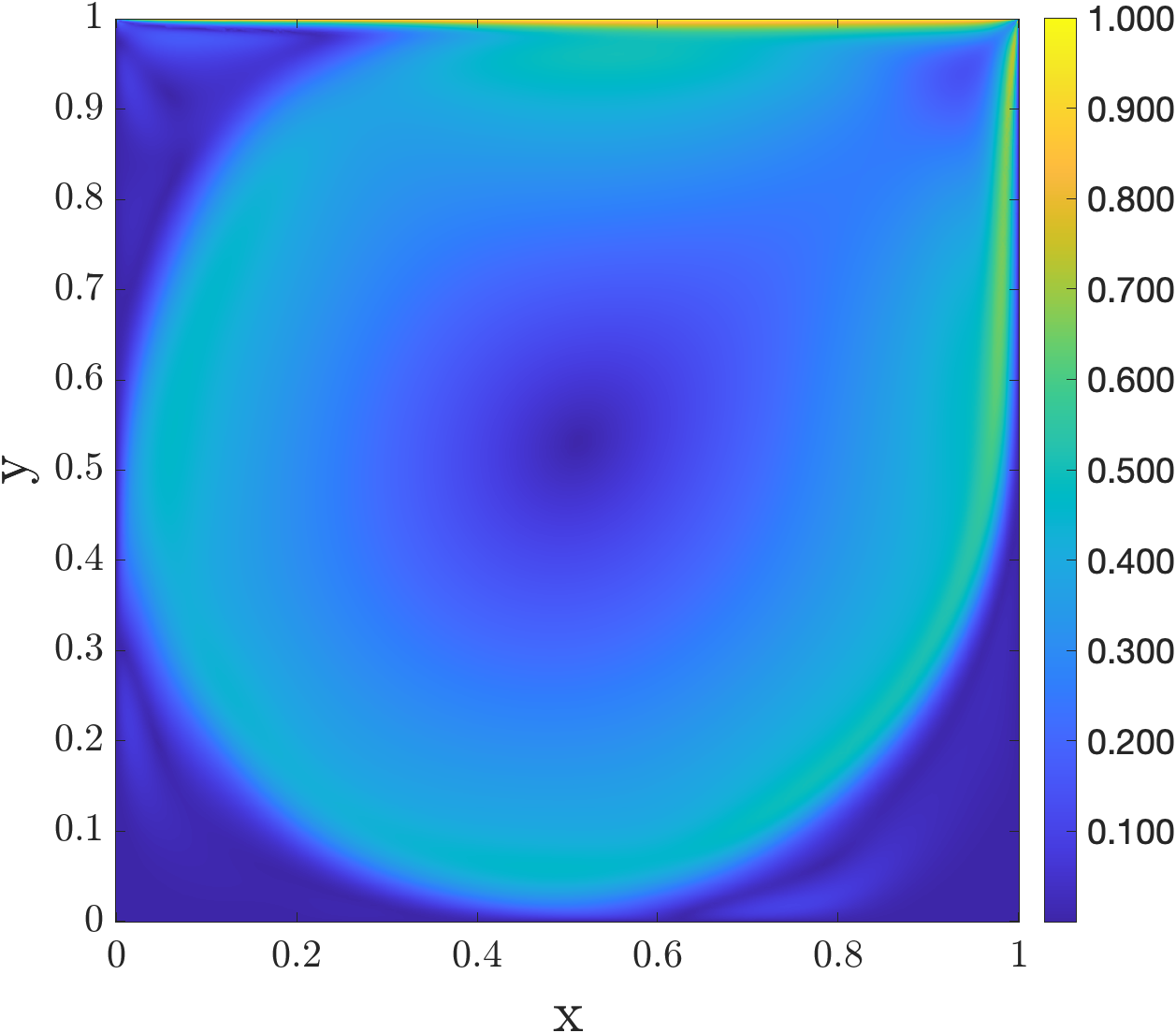}
        \caption{Full prediction.}
        \label{fig:fig1}
    \end{subfigure}\hfill
    \begin{subfigure}{0.4\textwidth}
        \centering
        \includegraphics[width=\linewidth]{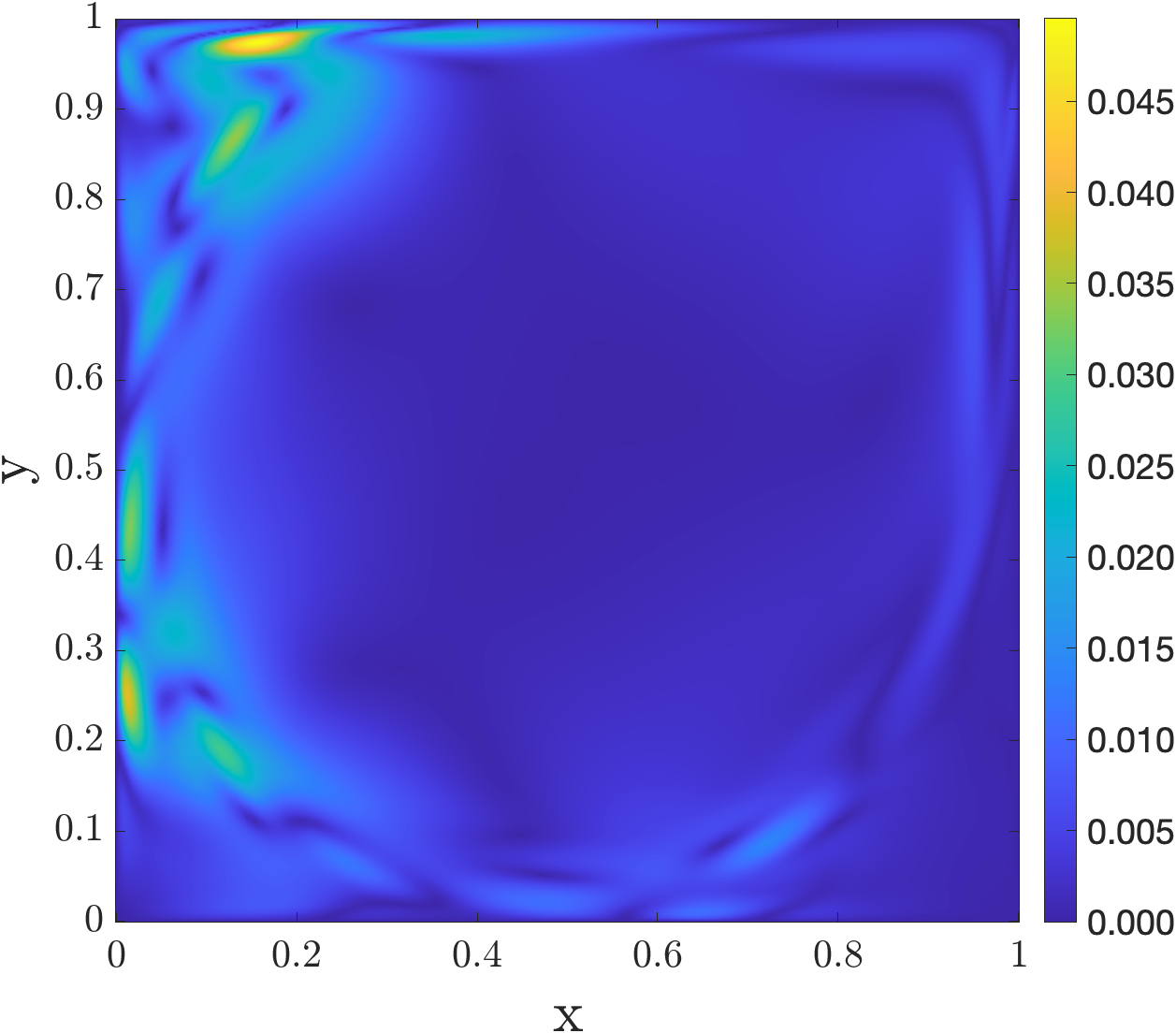}
        \caption{Predicted perturbation from fixed point.}
        \label{fig:fig2}
    \end{subfigure}
    \caption{Parametric SSM predictions for the lid-driven cavity flow at previously unseen Reynolds number.
    }
    \label{fig:comparison}
\end{figure}

\clearpage
\section{Discussion and Conclusion}

We have shown how SSM-based model reduction of a high-dimensional dynamical system around a fixed point undergoing a Hopf bifurcation can be justified by factoring in the location of certain resonances in the spectrum of the linear part. Locating these critical resonances, we showed how past approaches have suffered from locality issues in the pre-bifurcation parameter regime. As the order of these resonances increases closer to criticality, we have proved that the low-order Taylor coefficients in the SSM expansion and SSM-reduced dynamics persist on a larger parameter regime than the high-order coefficients. These findings generalize to any type of bifurcation of the fixed point, assuming the presence of at least one real negative eigenvalue.

We illustrated these results on a low-dimensional example in Hopf normal form, where analytic calculations showed the mechanism with which the loss of regularity at resonance appears. We also used a data-driven, parametric SSM approach to reduce lid-driven cavity flow to a 2D parametric SSM-model. This model accurately reproduced the full flow field at a large range of Reynolds numbers including the full transition from stable to periodic dynamics, while making a very precise prediction of the critical Reynolds number.

We believe that this work offers a rigorous guideline for modeling local bifurcations in a maximally extended parameter range using SSMs. Future work may target parametric SSM-models that capture other bifurcations, potentially of higher codimension, as the persistence results presented here can be directly generalized to those cases. Furthermore, our data-driven approach has the potential to capture and predict bifurcation phenomena from experimental data, a highly promising avenue for future research.

\section{Acknowledgments}
We are grateful to Giacomo Abbasciano for making his SSM interpolation code available.

\section{Author declarations}
The authors have no conflicts to disclose.

\appendix

\section{Appendix}

\subsection{Discretization}
\label{sec:discretization_navier_stokes}

After spatial discretization with finite differences, finite elements, finite volumes, or spectral methods, the discretized velocity and pressure fields are represented by vectors $\uu\in \R^{N}, \pp\in \R^{N_p}$.
Following \authcite{ascherComputerMethodsOrdinary1998}, the system \eqref{eq:Navier_Stokes} can be written as an index 2 differential algebraic equation (DAE)

\begin{align}
    \MM \dot{\uu} &= -\nu \KK \uu - \NN:(\uu \otimes \uu) - \CC \pp, \nonumber \\
    \CC^T \uu &= 0 \nonumber, \\
    \uu |_{\partial \Omega} &= \bb
    \label{eq:DAE}
\end{align}
where the matrices $\MM, \KK\in \R^{N\times N}, \CC\in \R^{N\times N_p}$, and tensor $\NN\in \R^{N\times N \times N}$ denote the mass matrix and the discretized diffusion, gradient and convection operators respectively. The vector $\bb \in \R^{N_b}$ contains the boundary condition.

To find the family of fixed points that changes stability type as the Reynolds number is increased, we solve for the fixed points of \eqref{eq:DAE} at low Reynolds numbers using a Picard iteration algorithm. This can then be used as an initial guess for a Newton method. This is faster to converge but on a smaller domain than the Picard iteration. The fixed point found in this fashion can then be used as an initial guess in the Picard iteration algorithm for a higher Reynolds number.

First, we expand the system around such a fixed point $(\uu_0,\pp_0)$, i.e. a time-independent solution to \eqref{eq:DAE}.
Letting $\ww=\uu-\uu_0$ and $\qq= \pp-\pp_0$, we can write \eqref{eq:DAE} as
\begin{align}
    \MM \dot{\ww} &= -\nu \KK\ww-\NN:(\ww \otimes \uu_0) - \NN:(\uu_0 \otimes \ww) - \NN:(\ww \otimes \ww) - \CC\qq, \nonumber \\
    \CC^T \ww &= 0, \nonumber \\
    \ww |_{\partial \Omega} &= 0,
    \label{eq:DAE_around_U0}
\end{align}
which we now view as a DAE on the interior points.

To study the stability of $(\uu_0, \pp_0)$, we note that the linear part of \eqref{eq:DAE_around_U0} can be written as
\begin{align}
    \EE \dot{\yy} = \AA \yy,
    \label{eq:linear_DAE}
\end{align}
where
\begin{align*}
    \yy = \begin{bmatrix}
        \ww \\
        \qq
    \end{bmatrix}, \qquad 
    \EE = \begin{bmatrix}
        \MM & 0 \\
        0 & 0
    \end{bmatrix}, \qquad 
    \AA = \begin{bmatrix}
        -\nu \KK\ww-\NN:(\uu_0 \otimes \cdot) - \NN:(\cdot \otimes \uu_0)& \CC \\
        \CC^T & 0
    \end{bmatrix}.
\end{align*}
The linear DAE \eqref{eq:linear_DAE} gives rise to generalized eigenproblems of the form $\AA \vv = \lambda \EE \vv$, which can be solved efficiently using a Krylov-Schur algorithm (\authcite{stewartKrylovSchurAlgorithmLarge2002}).

Next, we remove the pressure from the momentum equation using a Schur complement approach (\authcite{simoMixedHybridFinite1993}, \authcite{elmanFiniteElementsFast2014}). For this we define
\begin{align*}
    \RR(\ww;\nu) &= -\nu \KK\ww-\NN:(\uu_0 \otimes \ww) - \NN:(\ww \otimes \uu_0) - \NN:(\ww \otimes \ww),
\end{align*}
and differentiate the discretized continuity equation to obtain
\begin{align}
    \begin{bmatrix}
           \MM & \CC \\
           \CC^T & 0
    \end{bmatrix}
    \begin{bmatrix}
           \dot{\ww} \\
           \qq
    \end{bmatrix}
    = 
    \begin{bmatrix}
           \RR(\ww;\nu) \\
           0
    \end{bmatrix}.
    \label{eq:matrix_equation_ns}
\end{align}

Using the Schur complement $\SS= -\CC^T \MM^{-1} \CC$ of the block matrix above and assuming its invertibility\footnote{This depends on the spatial discretization. For our P2-P1 Taylor-Hood finite elements, this is true (see \authcite{simoMixedHybridFinite1993}).}, the matrix equation \eqref{eq:matrix_equation_ns} can be inverted to yield
\begin{align*}
    \begin{bmatrix}
           \dot{\ww} \\
           \qq
    \end{bmatrix}
    = 
    \begin{bmatrix}
           \MM^{-1} + \MM^{-1}\CC \SS^{-1}\CC^T \MM^{-1} & \MM^{-1}\CC \SS^{-1} \\
           -\SS^{-1}\CC^T \MM^{-1} & \SS^{-1}
    \end{bmatrix}
    \begin{bmatrix}
           \RR(\ww; \nu) \\
           0
    \end{bmatrix}.
\end{align*}
The decoupled momentum equation can now be written as 
\begin{align}
    \MM \dot{\ww} = \PP \RR(\ww;\nu),
    \label{eq:ODE_w}
\end{align}
where the projection $\PP= I - \CC[\CC^T \MM^{-1} \CC]^{-1}\CC^T \MM^{-1}$ can be seen as a disctretization of the Leray projection.

As a consequence, the evolution of the velocity perturbations $\ww$ in the discretized Navier-Stokes equations around the fixed point \eqref{eq:DAE_around_U0} can be viewed as a trajectory of a finite-dimensional parametric dynamical system
\begin{align*}
    \dot{\xx} = \ff(\xx;\mu), \ \xx\in \R^{N}, \ \ff\in C^{\infty}(\R^{N}\times \R^+,\R^{N}), \ \mu \in \R^+,
\end{align*}
where we introduced $\mu = 1/\nu$. Because \eqref{eq:DAE_around_U0} describes the flow relative to $\uu_0(\mu)$, we have $\ff(0;\mu)=0$ for each $\mu \in \R^+$.

\subsection{Proof of Lemma \ref{lem:failure_uniform_nonresonance}}
\label{sec:proof_lemma_failure_uniform_nonresonance}
Suppose there exists an open neighborhood $V_0$ of $\mu_0$ on which the nonresonance conditions \eqref{eq:nonresonance_conditions_3D} hold uniformly. For an arbitrary but fixed $\mu^*\in V_0$,
we define
\begin{align*}
    m^* =\frac{1}{2}\floor{\frac{\nu(\mu^*)}{\alpha(\mu^*)}}.
\end{align*}
Then, the strict inequality 
\begin{align*}
    (2m^*+1)\alpha(\mu^*)<\nu(\mu^*)<2m^*\alpha(\mu^*)
\end{align*}
holds as we have no resonance at $\mu^*$ by assumption. Moreover, as we assumed the nonresonance conditions to hold uniformly on $V_0$ and because 
$\alpha, \nu \in C^0$, the strict inequality 
\begin{align*}
    (2m^*+1)\alpha(\mu)<\nu(\mu)<2m^*\alpha(\mu),
\end{align*}
has to hold for all $\mu \in V_0$ by continuity. However, this contradicts $\alpha(\mu)\xrightarrow{\mu\to \mu_0}0$ and $\nu(\mu)\xrightarrow{\mu\to \mu_0} \nu_0 <0$.
\qed

\subsection{Proof of Lemma \ref{lem:resonance_locations}}
\label{sec:proof_lemma_resonance_locations}

By assumption, we can write Taylor expansions for $\mu < \mu_0$ as
\begin{align*}
    \alpha(\mu) &= a(\mu_0 - \mu)^p + r_{\alpha}(\mu)(\mu_0 - \mu)^p, &\norm{r_\alpha}_{L^\infty(\mu_0-\delta,\mu_0)}
\xrightarrow{\;\delta\to0^+\;}0, \\
    \nu(\mu) &= \nu_0 + b(\mu_0 - \mu) + r_{\nu}(\mu)(\mu_0 - \mu), &\norm{r_\nu}_{L^\infty(\mu_0-\delta,\mu_0)}
\xrightarrow{\;\delta\to0^+\;}0.
\end{align*}
with $p, a,b,\nu_0 \in \R$, $p\geq 1$ and $a,\nu_0 < 0$.
Because $m(\lambda_1(\mu) + \lambda_2(\mu)) = 2m\alpha(\mu)\in \R$, the nonresonance conditions \eqref{eq:nonresonance_conditions_3D} will fail at the zeros of the function 
\begin{align*}
    \Tilde{H}(\mu,m) &= \nu(\mu) - 2m\alpha(\mu) \\
    &= \nu_0 - 2ma(\mu_0 - \mu)^p + b(\mu_0 - \mu) + 2mr_{\alpha}(\mu)(\mu_0 - \mu)^p + r_{\nu}(\mu)(\mu_0 - \mu).
\end{align*}
We rescale $m$ using 
\begin{align}
    s= m^{1/p}(\mu_0 - \mu),
    \label{eq:rescaling}
\end{align}
and define $\epsilon = m^{-1/p}$ to obtain
\begin{align*}
    H(s,\epsilon) = \nu_0 - 2as^p + \epsilon bs - 2r_{\alpha}(\mu_0-\epsilon s)s^p + \epsilon r_{\nu}(\mu_0-\epsilon s)s.
\end{align*}
Because of the uniformity of the remainder terms $r_{\alpha}, r_{\nu}$, we have
\begin{align*}
    \lim_{\epsilon \rightarrow 0^+}H(s,\epsilon) = H(s,0) = \nu_0 - 2as^p,
\end{align*}
uniformly on compact sets of $s$.
We calculate the root
\begin{align}
    H(s^*,0) &= 0, \qquad s^* = \left(\frac{\nu_0}{2a}\right)^{1/p}.
    \label{eq:uniform_convergence_H}
\end{align}
Let $I=[s^*-\delta, s^* + \delta]$ with $\delta$ small enough such that 
\begin{align}
    \partial_s H(s,0) &= -2pa s^{p-1}\neq 0, \ \forall s \in I.
    \label{eq:monotonicity_H0}
\end{align}
If we assume the remainder terms to be $C^1$ (which is always the case if the system is smooth enough), then Lemma \ref{lem:resonance_locations} follows from the implicit function theorem. In that case, the roots are also unique. (Without this assumption we have to argue slightly differently and lose uniqueness).
Because of the monotonicity implied by \eqref{eq:monotonicity_H0}, $\forall s_1, s_2 \in I, s_1 < s^* < s_2$, we have
\begin{align*}
    H(s_1,0)H(s_2,0) < 0.
\end{align*}
Because \eqref{eq:uniform_convergence_H} converges uniformly, for any such $s_1, s_2$, there exists an $\epsilon>0$ small enough such that
\begin{align*}
    H(s_1,\epsilon)H(s_2,\epsilon) \leq 
    H(s_1,0)H(s_2,0) + C(\epsilon)\left(\abs{H(s_1,0)} + \abs{H(s_2,0)}\right) + C(\epsilon)^2 <
    0,
\end{align*}
where 
\begin{align*}
    C(\epsilon) = \norm{H(\cdot \ ,\epsilon)-H(\cdot \ ,0)}_{L^\infty(I)}.
\end{align*}
By the intermediate value theorem, fixing $s_1, s_2$ yields the existence of an $s_\epsilon \in (s_1, s_2)$ for which $H(s_\epsilon,\epsilon)=0$ for $\epsilon>0$ small enough.
Because $I$ is compact, the minimum 
\begin{align*}
    M = \min_{s\in I}\abs{\partial_s H(s ,0)}
\end{align*} 
is attained. By the mean value theorem, we also have
\begin{align*}
    \abs{s_\epsilon - s^*} \leq \frac{1}{M} \abs{H(s_\epsilon,0)-H(s^*,0)} = \frac{1}{M} \abs{H(s_\epsilon,0)}
    = \frac{1}{M} \abs{H(s_\epsilon,0)-H(s_\epsilon,\epsilon)} \leq \frac{C(\epsilon)}{M}.
\end{align*}
Hence, any $s_\epsilon$ is $o(1)$-close to $s^*$ as $\epsilon \rightarrow 0$.
Inverting the rescaling \eqref{eq:rescaling} gives us the root locations 
\begin{align*}
    \mu_{2m} = \mu_0 - \left(\frac{\nu_0}{2am}\right)^{1/p} + o(m^{-1/p}).
\end{align*}
\qed

\subsection{Proof of Lemma \ref{lem:low_order_ssm_coeffs}}
\label{sec:proof_lemma_low_order_ssm_coeffs}
The statement of the Lemma follows directly from the proof of Theorem 1.1 in \authcite{cabreParameterizationMethodInvariant2003a}. 
\qed

The following discusses a special case in order to give intuition for the role of resonances and the recursive computation of SSM coefficients. 
A non-parametric version of this can be found in \authcite{hallerModelingNonlinearDynamics2025}, and a guide on how the general computations can be performed in practice can be found in \authcite{jainHowComputeInvariant2022}.

We write out system \eqref{eq:parametric_dynamical_system} as 
\begin{align*}
    \dot{\xx} = \AA(\mu)\xx + \ff_0(\xx;\mu), \qquad \AA(\mu)=D_{\xx}\ff(\xx_0(\mu),\mu), \qquad \ff_0(\xx;\mu)=\LandO(\norm{\xx}^2).
\end{align*}
For ease of exposition, we will assume $\AA(\mu)$ to be semisimple with disjoint eigenvalues for all $\mu$. This implies the existence of an eigenbasis that is $C^{r}$ in $\mu$ (\authcite{sibuyaGlobalPropertiesMatrices1965}). (For the more general case, see Remark \ref{rem:non_semisimple_case}).
Let $\SS_{\uu}(\mu)\in \R^{N\times 2}$ contain the real eigenvectors of $\AA(\mu)$ associated to $\lambda_{1,2}(\mu)$ and $\SS_{\vv}(\mu)\in \R^{N\times N-2}$ contain the real eigenvectors associated to $\lambda_{3}(\mu), \dots \lambda_{N}(\mu)$.

Aligning the coordinate system with this basis yields
\begin{align}
    \dot{\uu} &= \AA_{\uu}(\mu)\uu + \ff_{\uu}(\uu,\vv;\mu), \label{eq:dynamics_u}\\
    \dot{\vv} &= \AA_{\vv}(\mu)\vv + \ff_{\vv}(\uu,\vv;\mu) \nonumber ,
\end{align}
where 
\begin{alignat}{2}
\xx &= \SS_{\uu}(\mu)\uu + \SS_{\vv}(\mu)\vv,
&\qquad \ff_{\uu}, \ff_{\vv} &= \mathcal{O}(\|(\uu,\vv)\|^2),
\label{eq:decomposition_eigenvectors}\\
\AA(\mu)\SS_{\uu}(\mu) &= \SS_{\uu}(\mu)\AA_{\uu}(\mu),
&\qquad \AA_{\uu}(\mu) &= \operatorname{diag}(\lambda_1(\mu), \lambda_2(\mu)), \nonumber\\
\AA(\mu)\SS_{\vv}(\mu) &= \SS_{\vv}(\mu)\AA_{\vv}(\mu),
&\qquad \AA_{\vv}(\mu) &= \operatorname{diag}(\lambda_3(\mu), \dots, \lambda_N(\mu)).
\nonumber
\end{alignat}

Assuming that the nonresonance conditions \eqref{eq:nonresonance_conditions} hold, a 2D slow SSM can be written locally as
\begin{align}
    \vv = \hh(\uu;\mu) = \sum_{2\leq\abs{\kk}\leq 2m+1} \hh^{\kk}(\mu) \uu^{\kk} + o(\norm{\uu}^{2m+1}),
    \label{eq:expansion_h}
\end{align}
where $\kk = (k_1,k_2)$ is a multi-index, $|\kk| = k_1 + k_2$, and $\uu^\kk = u_1^{k_1} u_2^{k_2}$. 
Because of the invariance of SSMs, we can differentiate this relationship along trajectories to obtain the parametric invariance equation
\begin{align}
    \AA_{\vv}(\mu)\hh(\uu;\mu) + \ff_{\vv}(\uu, \hh(\uu;\mu);\mu) = D_{\uu} \hh(\uu;\mu)[\AA_{\uu}(\mu)\uu + \ff_{\uu}(\uu,\hh(\uu;\mu);\mu)].
    \label{eq:parametric_invariance_eq}
\end{align}
By inserting \eqref{eq:expansion_h} into \eqref{eq:parametric_invariance_eq} and equating powers, we obtain a set of cohomological equations
\begin{align}
    \LL_{\kk}(\mu) \hh^{\kk}(\mu) = \NN^{\kk}(\{\hh^\jj(\mu)\}_{|\jj|<|\kk|};\mu),
    \label{eq:cohomological_equation}
\end{align}
\begin{align*}
    \LL_{\kk}(\mu) = \text{diag}\Big(\lambda_3(\mu)-(k_1\lambda_1(\mu)+k_2\lambda_2(\mu)),\dots, \lambda_N(\mu)-(k_1\lambda_1(\mu)+k_2\lambda_2(\mu))\Big).
\end{align*}
where $\NN^{\kk}(\cdot, \mu)$ is a nonlinear function of the coefficients $\hh^\jj(\mu)$ of order less than $|\kk|$. The inhomogeneous linear system \eqref{eq:cohomological_equation} becomes recursively solvable for increasing multiindices. For each $\abs{\kk}$, all cohomological operators $\LL_{\kk}(\mu)$ are invertible if and only if the nonresonance conditions at order $\abs{\kk}$ hold. Due to our assumptions on the resonances appearing in $\Tilde{V}_{2m}$, this is the case for all $\mu \in \Tilde{V}_{2m}$ as long as $\abs{\kk} \leq 2m+1$.
This means that the coefficients $\hh^{\kk}(\mu)$ of order $\abs{\kk}\leq 2m+1$ can be uniquely obtained for every $\mu\in \Tilde{V}_{2m}$.  Since $\LL_\kk(\mu)$ depends $C^{r}$ smoothly on $\mu$ and the inhomogeneity includes only the coefficients $\hh^\jj(\mu)$ of order less than $|\kk|$, it holds that $\hh^{\kk}(\mu)$ depends $C^{r}$ smoothly on $\mu$ via induction over $|\kk|$.

By inserting \eqref{eq:expansion_h} into \eqref{eq:decomposition_eigenvectors} we get a parametrization of the SSM as a graph over the eigenspace $\SS_{\uu}(\mu)$
\begin{align*}
    \xx = \WW(\uu;\mu) = \SS_{\uu}(\mu)\uu + \sum_{2\leq\abs{\kk}\leq 2m+1} \WW^{\kk}(\mu) \uu^{\kk} + o(\norm{\uu}^{2m+1}),
\end{align*}
where $\WW^{\kk}(\mu) = \SS_{\vv}(\mu) \hh^{\kk}(\mu)$. The persistence and smoothness of $\hh^{\kk}(\mu)$ transfers to $\WW^{\kk}(\mu)$.
\begin{remark}
\label{rem:non_semisimple_case}
The above argument can be extended to the general (non-semisimple) case by using the $C^r$ block-diagonalization, that is guaranteed if the eigenvalues remain disjoint (\authcite{sibuyaGlobalPropertiesMatrices1965}).
This leads to a form equivalent to \eqref{eq:decomposition_eigenvectors}, where $\AA_{\uu}(\mu)$ and $\AA_{\vv}(\mu)$ are no longer diagonal, but retain the same spectrum.

Solving the invariance equation order by order now leads to a coupled system of cohomological equations of the form
\begin{align}
    \LL_{|\kk|}(\mu) \{\hh^\ii(\mu)\}_{|\ii|=|\kk|} = \NN^{|\kk|}(\{\hh^\jj(\mu)\}_{|\jj|<|\kk|};\mu),
    \label{eq:cohomological_equation_general}
\end{align}
at each order $|\kk|$.
The operator $\LL_{|\kk|}(\mu)$ is invertible if and only if all nonresonance conditions at order $|\kk|$ hold. 
The even more general case in which the eigenvalues do not have to be disjoint is discussed in \authcite{cabreParameterizationMethodInvariant2003a} using spectral projections. The cohomological equations then appear as inhomogeneous linear equations involving linear operators between Banach spaces, the spectra of which contain precisely the nonresonance conditions at order $\abs{\kk}$.
\end{remark}

\subsection{Proof of Lemma \ref{lem:low_order_red_dyn_coeffs}}
\label{sec:proof_lemma_low_order_red_dym_coeffs}
By inserting \eqref{eq:expansion_h} into \eqref{eq:dynamics_u}, we can write the dynamics reduced to the primary SSM $W(E;\mu)$ as
\begin{align}
    \dot{\uu} &= \AA_{\uu}(\mu)\uu + \ff_{\uu}(\uu,\hh(\uu;\mu);\mu).
    \label{eq:reduced_dynamics_SSM}
\end{align}
We split $\hh(\uu;\mu)$ as
\begin{align*}
    \hh(\uu;\mu) &= \Tilde{\hh}(\uu;\mu) + \GammaGamma(\uu; \mu), \\
    \Tilde{\hh}(\uu;\mu)&= \sum_{2\leq\abs{\kk}\leq 2m+1} \hh^{\kk}(\mu) \uu^{\kk}, \qquad  \GammaGamma(\uu; \mu) = o(\norm{\uu}^{2m+1}),
\end{align*}
therefore \eqref{eq:reduced_dynamics_SSM} can be written as
\begin{align*}
    \dot{\uu} &= \AA_{\uu}(\mu)\uu + \ff_{\uu}(\uu,\Tilde{\hh}(\uu;\mu);\mu) + \E (\uu;\mu), \\
    \E (\uu;\mu) &= \ff_{\uu}(\uu,\Tilde{\hh}(\uu;\mu)+\GammaGamma(\uu; \mu);\mu) -\ff_{\uu}(\uu,\Tilde{\hh}(\uu;\mu);\mu).
\end{align*}
By the mean-value inequality, we have
\begin{align}
    \norm{\E (\uu;\mu)} \leq \sup_{s \in [0,1]}\norm{D_{\vv}\ff_{\uu}(\uu ,\Tilde{\hh}(\uu ;\mu)+s\GammaGamma(\uu; \mu);\mu)} \norm{\GammaGamma(\uu; \mu)}.
    \label{eq:mean_value_inequality}
\end{align}
Because $\ff_{\uu}$ is at least quadratic, \eqref{eq:mean_value_inequality} yields 
\begin{align*}
    \E (\uu;\mu) = o(\norm{\uu}^{2m+2}).
\end{align*}
This means that the Taylor coefficients of the reduced dynamics up to order $2m+2$ persist $C^{r}$ smoothly in $\mu$ across a resonance of order $r_0>2m$.
\qed

\subsection{Error Metrics}
\label{sec:NMTE}
As an error metric, Section \ref{sec:lid_driven_cavity} uses the normalized mean trajectory error (NMTE), which has frequently been employed to assess the predictive quality of ROMs (\authcite{cenedeseDatadrivenModelingPrediction2022}, \authcite{hallerModelingNonlinearDynamics2025}).

Given the full, high-dimensional dataset $\mathcal{D} = \{\xx(t_1), \dots, \xx(t_m)\}$ centered around the fixed point, and the ROM prediction $\Tilde{\mathcal{D}} =\{\Tilde{\xx}(t_1), \dots, \Tilde{\xx}(t_m)\}$, we define
\begin{align*}
    \text{NMTE}(\mathcal{D},\Tilde{\mathcal{D}})= \frac{1}{m \max_j\norm{\xx(t_j)}}\sum_{i=1}^m \norm{\xx(t_i) - \Tilde{\xx}(t_i)}.
\end{align*}
For problems with limit cycles, this error metric can often be too pessimistic as it is phase-sensitive. For this reason, we also report the
the normalized mean amplitude error (NMAE)
\begin{align*}
    \text{NMAE}(\mathcal{D},\Tilde{\mathcal{D}}) = \frac{1}{m \, \frac{1}{m} \sum_{i=1}^m \norm{\xx(t_i)}} \sum_{i=1}^m \left| \norm{\xx(t_i)} - \norm{\Tilde{\xx}(t_i)} \right|,
\end{align*}
which quantifies the average relative difference in trajectory amplitude and is less phase-sensitive.
To asses the quality of the SSM itself without its reduced dynamics, we define the Reconstruction Error (RecError)
\begin{align*}
    \text{RecError} = \text{NMTE}(\mathcal{D},\Hat{\mathcal{D}}),
\end{align*}
where $\Hat{\mathcal{D}} =\{\Hat{\xx}(t_1), \dots, \Hat{\xx}(t_m)\}$, $\hat{\xx}(t_i)=\WW(\VV(\mu)^T\xx(t_i);\mu)$, i.e. the NMTE between the full data and the projected-then-reconstructed data.

\printbibliography
\end{document}